\documentclass[11pt]{article}

\usepackage{anyfontsize}
\usepackage{lmodern}
\usepackage{microtype}
{}% \selectfont needed

\usepackage{amssymb}
\usepackage{graphicx}
\usepackage{mathrsfs}
\usepackage{amsmath}
\usepackage{amsthm}
\usepackage{tikz-cd} %quivers
\usepackage{tikz}
\newcommand*\circled[1]{\tikz[baseline=(char.base)]{
  \node[shape=circle,draw,inner sep=1pt] (char) {#1};}}

\usepackage{amsfonts}
\usepackage{latexsym}
\usepackage{mathtools} %\dcases,  \DeclarePairedDelimiter\floor{\lfloor}{\rfloor}, \set
\mathtoolsset{showonlyrefs}
\usepackage[colorlinks=true, urlcolor=blue, linkcolor=blue, citecolor=blue]{hyperref}
\usepackage[
  backend=biber,
  style=alphabetic,
  sorting=anyt,
  maxnames=6,
  maxalphanames=6,
  giveninits=true,
  doi=false,
  isbn=false,
  url=false,
]{biblatex}
\DeclareFieldFormat{extraalpha}{#1}
\DeclareLabelalphaTemplate{
  \labelelement{
    \field[final]{shorthand}
    \field{label}
    \field[strwidth=3,strside=left,ifnames=1]{labelname}
    \field[strwidth=1,strside=left]{labelname}
  }
}

\numberwithin{equation}{section}
\newtheorem{theorem}{Theorem}[section]
\newtheorem{proposition}[theorem]{Proposition}
\newtheorem{lemma}[theorem]{Lemma}
\newtheorem{corollary}[theorem]{Corollary}

\theoremstyle{definition}
\newtheorem{definition}[theorem]{Definition}

\newtheorem{remark}[theorem]{Remark}

\theoremstyle{plain}
\newtheorem{claim}[theorem]{Claim}

\numberwithin{equation}{section}

\newcommand{\R}{\mathbb{R}}  % The real numbers.
\newcommand{\C}{\mathbb{C}}  % The complex numbers.
\newcommand{\Sph}{\mathbb{S}}  % The unit sphere.
\DeclareMathAlphabet{\mathbbold}{U}{bbold}{m}{n}
\newcommand*{\1}{\mathbbold{1}}

\DeclareMathOperator{\dist}{dist} % The distance.
\DeclareMathOperator{\Id}{Id}
\DeclareMathOperator{\Int}{Int}
\DeclareMathOperator{\ind}{ind}
\DeclareMathOperator{\nul}{nul}
\DeclareMathOperator{\supp}{supp}
\DeclareMathOperator{\Area}{Area}
\DeclarePairedDelimiter\floor{\lfloor}{\rfloor}
\DeclarePairedDelimiterX\set[2]{\{}{\}}{\,#1 \;\delimsize\vert\; #2\,}
\newcommand*{\dv}[1]{dv_{#1}} % volume measure
\newcommand*{\G}{\Gamma} % action group
\newcommand*{\ds}[1]{ds_{#1}} % hypersurface measure
\newcommand*{\abs}[1]{\left|#1\right|} % absolute value
\newcommand*{\norm}[1]{\left\|#1\right\|} % norm
\newcommand*{\brr}[1]{\left(#1\right)} % bracket round
\newcommand*{\brs}[1]{\left[#1\right]} % bracket square
\newcommand*{\brc}[1]{\left\{#1\right\}} % bracket curly
\newcommand*{\brt}[1]{\left\langle #1\right\rangle } % bracket triangle
\mathchardef\mhyphen="2D %the short hyphen in math formulas
\newcommand{\mres}{\mathbin{\vrule height 1.6ex depth 0pt width
0.13ex\vrule height 0.13ex depth 0pt width 1.3ex}} % a sign for a measure restricted on a set
\renewcommand{\epsilon}{\varepsilon}
\renewcommand{\phi}{\varphi}
\newcommand\quot[2]{
  \mathchoice
      {% \displaystyle
          \text{\raisebox{.2em}{$#1$}\Big/\raisebox{-.2em}{$#2$}}%
      }
      {% \textstyle
          #1/#2
      }
      {% \scriptstyle
          #1\,/\,#2
      }
      {% \scriptscriptstyle
          #1\,/\,#2
      }
}
\makeatletter
\newcommand*{\oset}[3][0.45ex]{%
  \mathrel{\mathop{#3}\limits^{
    \vbox to#1{\kern-2\ex@
    \hbox{$\scriptstyle#2$}\vss}}}}
\makeatother
\begin{document}

%%
%% The title of the paper goes here.  Edit to your title.
%%
\author{Denis Vinokurov}
\title{Conformal optimization of eigenvalues on surfaces with symmetries}
\date{}
\maketitle
\begin{abstract}
  Given a conformal action of a discrete group on a Riemann surface,
  we study the maximization of Laplace and Steklov
  eigenvalues within a conformal class, considering metrics invariant under the group action.
  We establish natural conditions for the existence and regularity
  of maximizers. Our method simplifies the previously
  known techniques for proving existence and regularity results
  in conformal class optimization.
  Finally, we provide a complete solution to the equivariant
  maximization problem for Laplace eigenvalues on the
  sphere and Steklov eigenvalues on the disk,
  resolving open questions posed by Arias-Marco et al.~(2024)
  regarding the sharpness of the Hersch–Payne–Schiffer inequality and the maximization of Steklov eigenvalues by the standard
  disk among planar simply connected domains with
  $n\text{-rotational}$ symmetry.
\end{abstract}

%%
%% LaTeX can automatically make a table of contents.  This is done by
%% uncommenting the following:
%%
\tableofcontents
% \newpage

\section{Introduction and main results}
Let $(\Sigma, [g])$ be a compact surface with a conformal class
$[g]$. The Laplacian $\Delta_g\colon C^\infty(\Sigma) \to C^\infty(\Sigma)$
of the metric $g$ is defined by the formula
\begin{equation}
  \Delta_g \phi = -\operatorname{div}_g \operatorname{grad}_g \phi.
\end{equation}
When $\Sigma$ is closed, i.e.
 $\partial \Sigma = \varnothing$, we consider the classical
 eigenvalue problem,
 \begin{equation}
  \Delta_g \phi = \lambda \phi.
 \end{equation}
 If the boundary is nonempty, we replace it with the Steklov problem,
 \begin{equation}
  \begin{cases}
    \Delta_g \phi = 0
    \\ \partial_\nu \phi = \sigma \phi.
  \end{cases}
 \end{equation}
 In both cases, the eigenvalues are real nonnegative numbers
converging to $\infty$, so they can be written in nondecreasing
order, counting multiplicities:
 \begin{gather}
  0 = \lambda_0(g) < \lambda_1(g) \leq
  \lambda_2(g) \leq \cdots\nearrow \infty,
  \\ 0 = \sigma_0(g) < \sigma_1(g) \leq
  \sigma_2(g) \leq \cdots\nearrow \infty.
 \end{gather}
 Consider the following
 two optimization problems,
 \begin{align}
  \Lambda_k(\Sigma, [g]) &= \sup_{\tilde{g}\in[g]}
  \overline{\lambda}_k(\tilde{g}),
  \ &\text{where } \
  \overline{\lambda}_k(g) &:=
  \lambda_k(g) \operatorname{Area}_{g}(\Sigma),
  \\ \mathfrak{S}_k(\Sigma, [g]) &= \sup_{\tilde{g}\in[g]}
  \overline{\sigma}_k(\tilde{g}),
  \ &\text{where } \
  \overline{\sigma}_k(g) &:=
  \sigma_k(g) \operatorname{Length}_{g}(\partial\Sigma).
 \end{align}
 These are two classical examples of geometric eigenvalue
 optimization problems. We refer to~\cite{Karpukhin-Nadirashvili-Penskoi-Polterovich:2022:existence,
 Karpukhin:2021:index-of-min-shperes, Colbois-Girouard-Gordon-Sher:2024:recent-develop-on-steklov}
 and references therein
 for a detailed survey of recent developments on the problems,
 see  also \cite*{Kim:2022:second-sphere-eigenval,
 Karpukhin-Metras-Polterovich:2024:dirac-eigenval-opt,
 Perez-Ayala:2021:conf-laplace-sire-xu-norm,
 Perez-Ayala:2022:extr-metrics-paneitz-operator,
 Gursky-Perez-Ayala:2022:2d-eigenval-conform-laplacian}
 for some related eigenvalue optimization results.

 It is known (e.g.,~\cite{Karpukhin-Nadirashvili-Penskoi-Polterovich:2022:existence,
 Petrides:2019:max-steklov-eigenval-on-surfaces}) that:
 \begin{proposition} Given a compact connected surface $(\Sigma, [g])$, there exists a smooth harmonic map
  $\Phi\colon \Sigma \to \mathbb{S}^n$ (the case $\Sigma = \varnothing$),
  or a smooth free boundary harmonic map
  $\Psi\colon \Sigma \to \mathbb{B}^n$ (the case $\Sigma \neq \varnothing$)
  such that one of the following equalities holds
  \begin{equation}
    \Lambda_k(\Sigma, [g]) = \max_{b \leq k}
    \brc{\overline{\lambda}_{k-b}(\abs{d\Phi}_g^2 g) + 8\pi b}
  \end{equation}
  or
  \begin{equation}
    \mathfrak{S}_k(\Sigma, [g]) = \max_{b \leq k}
    \brc{\overline{\sigma}_{k-b}(\abs{\partial_\nu\Psi} g) + 2\pi b},
  \end{equation}
  where $\abs{\partial_\nu\Psi}$ is extended arbitrarily in the
  interior. The quantities $8\pi b$ and $2\pi b$ appear as a result
  of attaching $b$ spheres or disks.
\end{proposition}
In particular, the suprema are achieved
on a metric possibly with a finite number of conical singularities,
or they are achieved on $\Sigma$ together with several
spheres $\Sph^2$ or disks $\mathbb{D}^2$ attached. Such a phenomenon
is called \emph{bubbling}, and the corresponding
$\Sph^2$ and $\mathbb{D}^2$ are referred as \emph{bubbles}.
Since bubbles produce an obstruction for the existence of a
maximizer, one can try to deal with them
by adding a conformal action of a finite group $\G$ and
consider only $\G$-invariant metrics. Then bubbles can only
appear as whole orbits of $\G$, which makes it more unlikely
to happen.

In the paper~\cite{Karpukhin-Kusner-McGrath-Stern:2024:equiv-eigenval-optim},
equivariant optimization was used
to produce new examples of minimal surfaces embedded in $\Sph^3$ and
to show that every compact oriented surface with boundary
can be realized as an embedded free boundary minimal surface in $\mathbb{B}^3$,
a problem that had remained open before.

Let us denote the corresponding $\G$-invariant suprema
as $\Lambda_k(\Sigma, [g]; \G)$ and
$\mathfrak{S}_k(\Sigma, [g];\G)$. In the present paper, we
formulate the conditions for the existence of such maximizers.
When it happens, we prove that the metric comes from a (smooth)
$\G$-equivariant harmonic map to a sphere $\Sph^n$
in the closed case,
and in the Steklov case, it comes from a (smooth) free boundary
$\G$-equivariant harmonic map to a ball $\mathbb{B}^n$. As an application,
we compute $\Lambda_k(\Sph^2, [g_{\Sph^2}];\G)$ and
$\mathfrak{S}_k(\mathbb{D}^2, [g_{\R^2}];\G)$ for all the possible
actions of finite groups $\G$.

\subsection{Regularity and existence of the maximizers}

It is natural that the set of
bubble points is preserved by the group. Let us denote by $\G \!\cdot p$ the orbit of a point
$p \in \Sigma$, and
introduce the semigroup of natural numbers
\begin{equation}\label{eq:N_G}
  \mathbb{N}_\G(\Sigma) = \brc{0}\cup\brc{\abs{\G \!\cdot\! p_1}+\cdots+\abs{\G \!\cdot\! p_n}\colon p_i\in\Sigma}
\end{equation}
generated by all possible cardinalities $\abs{\G \!\cdot\! p}$ of the orbits under the action of $\G$.
We say that a map $\Phi\colon \Sigma \to \R^n$ is $\G$-equivariant
if there exists a right representation $r\colon\G \to \mathrm{O}\brr{\mathbb{R}^n}$
such that $\gamma^*\Phi = r(\gamma)\Phi$ for any $\gamma \in \G$. Let us define
\begin{equation}
  \Lambda_k(\Sigma, [g];\G)
  := \sup \set*{\overline{\lambda}_k(\tilde{g})}{\tilde{g} \in [g],\
  \tilde{g} \text{ is } \G\text{-invariant}}.
\end{equation}
\begin{theorem}\label{thm:main}
  Let $\brr{\Sigma, [g]}$ be a closed connected surface with a conformal
  action of a discrete group $\G$. Then
  \begin{equation}
    \Lambda_k(\Sigma, [g]; \G) \geq \max_{
      \substack{b\in\mathbb{N}_\G(\Sigma) \\ 1\leq b \leq k}}
    \brc{\Lambda_{k-b}(\Sigma, [g]; \G) + 8\pi b}.
  \end{equation}
  If the inequality is strict, there exists a nontrivial smooth $\G$-equivariant
  harmonic map $\Phi\colon \Sigma \to\mathbb{S}^n$
  such that $\lambda_{k}(\abs{d\Phi}_g^2 g) = 1$ and
  \begin{equation}
    \Lambda_k(\Sigma, [g]; \G) = \overline{\lambda}_{k}(\abs{d\Phi}_g^2 g)
    = \int_\Sigma \abs{d\Phi}_g^2 dv_g.
  \end{equation}
\end{theorem}
\begin{remark}
  Since $\Phi$ is harmonic, the metric
  $|d\Phi|_g^2 g$
  is smooth away from a finite set of conical singularities~\cite{Helein:2002:harm-maps-and-moving-frames}.
  The eigenvalues $\lambda_k(|d\Phi|_g^2 g)$ are interpreted in a measure-theoretical
  sense~\eqref{eq:eigenval-of-meas}.
\end{remark}
\begin{corollary}\label{cor:non-babl}
  If $\abs{\G \!\cdot\! p} > k$ for any $p \in \Sigma$,
  there exists a nontrivial smooth $\G$-equivariant harmonic map $\Phi\colon \Sigma \to
    \mathbb{S}^n$ such that $\lambda_{k}(\abs{d\Phi}_g^2 g) = 1$ and
  \begin{equation*}
    \Lambda_k(\Sigma, [g];\G) = \overline{\lambda}_k(\abs{d\Phi}_g^2 g).
  \end{equation*}
\end{corollary}
\medskip
In the case of a nonempty boundary, we define
\begin{equation}
  \mathfrak{S}_k(\Sigma,[g];\G)
  := \sup \set*{\overline{\sigma}_k(\tilde{g})}{\tilde{g} \in [g],\
  \tilde{g} \text{ is } \G\text{-invariant}}.
\end{equation}
\begin{theorem}\label{thm:main-stek}
  Let $\brr{\Sigma, [g]}$ be a compact connected surface with
  boundary $\partial \Sigma \neq \varnothing$ and a conformal
  action of a discrete group $\G$. Then
  \begin{equation}
    \mathfrak{S}_k(\Sigma, [g];\G) \geq \max_{
      \substack{b\in\mathbb{N}_\G(\partial\Sigma) \\ 1\leq b \leq k}}
    \brc{\mathfrak{S}_{k-b}(\Sigma, [g];\G) + 2\pi b}.
  \end{equation}
  If the inequality is strict, there exists a nontrivial smooth $\G$-equivariant
  free boundary harmonic map $\Phi\colon \Sigma \to\mathbb{B}^n$
  such that $\sigma_k(\abs{\partial_\nu \Phi}^2 g) = 1$ and
  \begin{equation}
    \mathfrak{S}_{k}(\Sigma, [g];\G) = \overline{\sigma}_k(\abs{\partial_\nu \Phi}^2 g)
     = \int_\Sigma \abs{d\Phi}_g^2 dv_g,
  \end{equation}
  where $\abs{\partial_\nu \Phi}^2g$ is positively extended to a metric
  on the whole $\Sigma$.
\end{theorem}
\begin{remark}
  Since $\Phi$ is a free boundary harmonic map, one has
  $|\partial_\nu \Phi|^2 > 0$ on the boundary, see~\cite[Lemma~2]{Karpukhin-Metras:2022:higher-dim}.
  At the same time, the Steklov problem has the property that
  the eigenvalues $\overline{\sigma}_k(|\partial_\nu \Phi|^2 g)$ do not
  depend on a positive extension of $|\partial_\nu \Phi|^2$ in the interior of $\Sigma$.
\end{remark}
\begin{corollary}\label{cor:stek-non-babl}
  If $\abs{\G \!\cdot\! p} > k$ for any $p \in \partial\Sigma$,
  there exists a nontrivial, smooth $\G$-equivariant free boundary harmonic map
  $\Phi\colon \Sigma \to \mathbb{B}^n$ such that $\sigma_k(\abs{\partial_\nu \Phi}^2 g) = 1$ and
  \begin{equation*}
    \mathfrak{S}_k(\Sigma, [g];\G) = \overline{\sigma}_k(\abs{\partial_\nu \Phi}^2 g).
  \end{equation*}
\end{corollary}
\begin{remark}
  Our results remain valid even in the absence of a group action on $\Sigma$
  (i.e. when $\G = \brc{e}$). This leads to simplified proofs of previously known existence and regularity results for
  $\Lambda_k$ and $\mathfrak{S}_k$ (cf.~\cite{Karpukhin-Nadirashvili-Penskoi-Polterovich:2022:existence,
  Petrides:2018:exist-of-max-eigenval-on-surfaces,
  Petrides:2019:max-steklov-eigenval-on-surfaces}). These simplifications
  mainly concern the bubbling-tree construction
  (see Remark~\ref{rem:bubbling-cons}) and a shorter proof of $H^1$-convergence away from
  $\emph{unstable}$ (a.k.a., $\emph{bad}$) points
  (see Remark~\ref{rem:h2-conv-simple}).
\end{remark}

\subsection{Applications}
We are going to compute $\Lambda_k(\Sph^2,[g];\G)$ and
$\mathfrak{S}_k(\mathbb{D}^2,[g];\G)$ for any finite group
admitting a conformal action on these surfaces.
Up to a diffeomorphism, both $\Sph^2$ and $\mathbb{D}^2$ each
have only one conformal class, so the dependency of $\Lambda_k$ and
$\mathfrak{S}_k$ on $[g]$ can be omitted.

Let $g_{\Sph^2}$ denote the standard round metric on $\Sph^2$ and
$G$ be a finite subgroup of $SO(\R^3)$. Then all such groups
are classified as follows,
\begin{equation}\label{eq:classif-of-so3}
  G \in \brc{\mathbb{Z}_n, D_n, T, O, I},
\end{equation}
where $\mathbb{Z}_n$ is the group of rotations of order $n$,
$D_n$ is the dihedral group of a regular $n$-gon,
and the last three are the (rotational) tetrahedral,
octahedral and icosahedral groups correspondingly.
\begin{proposition}
  Let $\G$ be a finite subgroup of the conformal transformations
  of $\Sph^2$, i.e. $[\gamma^* g_{\Sph^2}] = [g_{\Sph^2}]$. Then
  up to a conformal automorphism of the sphere,
  $\G$ is a finite subgroup of $O(\R^3)$. In particular,
  one of the following situations is realized:
  \begin{itemize}
    \item $\G \approx G$, $\pi \in \hom(\G, \mathbb{Z}_2)$, and
    $\Id\times\pi\colon\G \hookrightarrow SO(\R^3)\times\mathbb{Z}_2$,
    \item $\G = G \times \mathbb{Z}_2 \subset SO(\R^3)\times\mathbb{Z}_2$,
  \end{itemize}
  where $G$ is from~\eqref{eq:classif-of-so3},
  $O(\R^3) \approx SO(\R^3)\times\mathbb{Z}_2$ and
  $(\Id, -1)$ corresponds to the antipodal map $x \mapsto -x$.
\end{proposition}
\begin{proof}
  The group $\G$ is a subgroup of the
  Möbius group $M(\Sph^2)$ generated by
  inversions in $\Sph^2 \approx \mathbb{C}\cup\brc{\infty}$.
  Since $\G$ is finite, it is conjugate to a subgroup of $O(\R^3)$; see, e.g.,~\cite[Theorem~5.5.2]{Ratcliffe:2019:hyperbolic-manifolds}.
  A sketch of the proof goes as follows.
  The Möbius transformations of $\R^3\cup\brc{\infty}$ that fix
  $\mathbb{B}^3$ are uniquely determined by their action on
  $\Sph^2$, and any element of $M(\Sph^2)$ extends to
  a Möbius transformation of $\R^3$. Thus, we have the following isomorphisms,
  $M(\Sph^2) \approx M(\mathbb{B}^3) \approx \operatorname{Iso}(\mathbb{H}^3)$,
  where the last one is simply the Poincaré model of the hyperbolic space
  $\mathbb{H}^3$. Then it is easy to see that the finite group $\G$
  has a fixed point $\frac{v}{\norm{v}_{\mathbb{H}}}$, where
  $v = \sum_{\gamma \in \G} \gamma(v_0)$ for some $v_0\in \mathbb{H}^3$.
  The stabilizer of a point under the action of $M(\mathbb{B}^3)$
  is conjugated to $O(\R^3)$.

  Finally, the proposition follows after classifying
  all the subgroups of the direct product
  $G_1\times G_2 = SO(\R^3)\times\mathbb{Z}_2$.
\end{proof}
Let $\G_+ = \G \cap SO(\R^3)$ be the orientation
preserving transformations of $\G$.
\begin{corollary}
  Let $\G$ be a finite subgroup of the conformal transformations
  of $\mathbb{D}^2$. Then either $\G = \mathbb{Z}_n \subset SO(\R^2)$
  is the group of rotations of order $n$,
  or $\G = D_n \subset O(\R^2)$ is the plane Dihedral group.
\end{corollary}
\begin{proof}
  In the case of the closed unit disk, the classification~\eqref{eq:classif-of-so3}
  implies $\G_+ = \mathbb{Z}_n$ since the other cases do not preserve
  $\mathbb{D}^2$ (e.g.,
  $D_n \subset SO(\R^3)$ interchanges the half-spheres
  if the disk is identified with one of them). By the same reason,
  $(\Id, -1) \notin \G$. So, $\mathbb{Z}_n = \ker \pi$, where
  $\pi \colon\G\to\mathbb{Z}_2$, and
  either $\G = \mathbb{Z}_n$ and
  $\pi$ is trivial, or
  $\G = \brc{(g,\pi(g))\in D_{n} \times \mathbb{Z}_2}$, which is just
  the plane Dihedral group (still denoted by $D_{n}$).
\end{proof}

\subsubsection{Conformal actions on \texorpdfstring{$\mathbb{D}^2$}{Lg}}

Recall that the Steklov spectrum of the unit disk looks as follows
  \begin{equation}\label{eq:disk-spec}
    \brc{0, 1, 1, 2, 2, 3, 3, \cdots }.
  \end{equation}
  So, $\overline{\sigma}_k(g_{\R^2}) = 2\pi\floor*{\frac{k+1}{2}}$.
\begin{theorem}\label{thm:symmetric-disk}
  Let $\G$ be the either the (plane) Dihedral group $D_n$ or the
  cyclic group $\mathbb{Z}_n$ generated by rotations of order $n$. Then
  \begin{equation}
    \mathfrak{S}_k(\mathbb{D}^2;\G) = 2\pi\brr{mn
    + \floor*{\frac{r+1}{2}}}
    = mn\,\overline{\sigma}_1(g_{\R^2}) + \overline{\sigma}_r(g_{\R^2}),
  \end{equation}
  where $k = mn + r$, $r < n$ and $\floor{x}$ denotes the
  greatest integer, lower than $x$. The supremum is
  attained on a disjoint union of $mn + \1_{r \neq 0}$ Euclidean disks.
\end{theorem}
\begin{remark}
  If $k \geq n$ and $k > 1$, one has a strict inequality $\bar{\sigma}_k(g)
  < \mathfrak{S}_k(\mathbb{D}^2;\G)$ for any $\G$-invariant
  metric $g$.
  In particular, when $n=1$, we obtain the strict inequality
  $\bar{\sigma}_k(g) < \mathfrak{S}_k(\mathbb{D}^2)$.
  The proof is essentially based on~\cite[Theorem~2.1]{Fraser-Schoen:2015:unique-free-disks}
  and the fact that the standard disk is the only maximizer for
  $\overline{\sigma}_k(g)$ among its branched coverings (see
  Lemma~\ref{lem:best-cover-disk}).
  The strictness of the inequality was already stated in~\cite[Theorem~2.3]{Fraser-Schoen:2020:steklov-unions}.
  Here we add the details for the case of a branched covering of the disk.
\end{remark}
\begin{remark}
  Notice that the configuration of $mn + 1$ disks is not unique. For example,
  in the case when $n = 3$ and $k = 3 + 2$, we can rescale
  two disjoint disks to obtain
  $\overline{\sigma}_3(g_{\R^2}) + \overline{\sigma}_2(g_{\R^2})
  = 3\,\overline{\sigma}_1(g_{\R^2}) + \overline{\sigma}_2(g_{\R^2})$
\end{remark}

Some properties of $\mathfrak{S}_k(\mathbb{D}^2;\mathbb{Z}_n)$
were studied in~\cite{Hassannezhad:2024:possibly-hidden-symmetry-steklov}.
As a consequence of Theorem~\ref{thm:symmetric-disk},
we answer \cite[Question~4.4 and Question~4.5]{Hassannezhad:2024:possibly-hidden-symmetry-steklov}.
\begin{corollary}
For an integer $n \geq 1$, the standard disk realizes $\mathfrak{S}_k(\mathbb{D}^2;\mathbb{Z}_n)$
if and only if $k \leq n-1$.
\end{corollary}
\begin{corollary}
For a planar simply connected domain with $n$-rotational symmetry,
the Hersch–Payne–Schiffer inequality is sharp
only when $k \equiv 0$ or $k \equiv 1 \ mod \ n$, i.e.
\begin{equation}
  \mathfrak{S}_k(\mathbb{D}^2;\mathbb{Z}_n) < 2\pi k
  \quad \text{if} \ (k \ mod \ n) \notin \brc{0,1}.
\end{equation}
\end{corollary}

\subsubsection{Conformal actions on \texorpdfstring{$\mathbb{S}^2$}{Lg}}
The cases when $\G_+ \in \brc{\mathbb{Z}_n, D_n}$ are not very interesting
since the group $\G$ fixes the axis of rotation, so $\mathbb{N}_\G \supset 2\mathbb{N}$
and $\Lambda_k(\Sph^2; \G) = \Lambda_k(\Sph^2) = 8\pi k$.

If $\G_+ \in \brc{T, O, I}$, we notice that $\pi \in \hom(\G, \mathbb{Z}_2)$
is nontrivial only when $\G_+ = T \subset O \approx \G$ because neither $T$, nor $I$ have subgroups
of index $2$. In this case, $\G = T_d$ is the full tetrahedral group.
Thus, it remains to consider the following options:
\begin{enumerate}
  \item $\G \in \brc{T, T_d} \approx \brc{A_4, S_4}$,
  where $T_d$ is the full tetrahedral group;
  \item $\G \in \brc{T_h, O, O_h} \approx
  \brc{A_4\times\mathbb{Z}_2, S_4, S_4\times\mathbb{Z}_2}$,
  where $T_h$ is the pyritohedral group, and $O_h$ is the full
  octahedral group.
  \item $\G \in \brc{I, I_h} \approx
  \brc{A_5, A_5\times\mathbb{Z}_2}$,
  where $I_h$ is the full icosahedral group.
\end{enumerate}
Recall that
$\overline{\lambda}_k(g_{\Sph^2}) \in \set*{4m(m+1)\pi}{m \in \mathbb{N}}$. Namely,
\begin{equation}\label{eq:spher-harmonics}
  Spec(\Delta_{g_{\Sph^2}}) = \set*{m(m+1)}{m \in \mathbb{N}}
  \ \text{and} \
  \operatorname{mult}\brs{m(m+1)} = 2m+1.
\end{equation}
The following theorem basically says that there are
only 2 possibilities to realize $\Lambda_k(\Sph^2; \G)$. It happens
either on the round metric or on a disjoint union of those.

\begin{theorem}\label{thm:equiv-eigenval-sphere}
  Let $\G$ be as above and conformally act on $\Sph^2$. Then
  \begin{equation}
    \Lambda_k(\Sph^2; \G) = 8\pi k' + \overline{\lambda}_s(g_{\Sph^2}),
  \end{equation}
  where $k = k' + s$ and $k' = \max \brr{\mathbb{N}_\G \cap [0, k]}$.
  The supremum is attained on a disjoint union of $k' + \1_{s \neq 0}$
  spheres. If $k' = 0$, the equality $\overline{\lambda}_k(g) = \Lambda_k(\Sph^2; \G)$
  is achieved on a $\G$-invariant metric $g$ possibly with conical singularities
  if and only if $g = g_{\Sph^2}$. For $k' \geq 1$,
  there is no maximizing metric.

  In particular,
  if $\G \in \brc{T, T_d}$, then
  \begin{align}
    \Lambda_k(\Sph^2; \G) &= 8\pi  &\text{ when } &k \in \brc{1,2,3},
    \\ \Lambda_k(\Sph^2; \G) &= 8\pi k &\text{ when } &k \geq 4;
  \end{align}
  if $\G \in \brc{T_h, O, O_h}$, then
  \begin{align}
    \Lambda_k(\Sph^2; \G) &= \overline{\lambda}_k(g_{\Sph^2}) &\text{ when } &1\leq k \leq 5,
    \\ \Lambda_k(\Sph^2; \G) &= 8\pi\cdot9 &\text{ when } &k \in \brc{10,11},
    \\ \Lambda_k(\Sph^2; \G) &= 8\pi k &\text{ when } &k \in [6,9] \cup [12, \infty);
  \end{align}
  If $\G \in \brc{I, I_h}$, then $\mathbb{N}_\G = \set*{12\alpha+20\beta+30\gamma}
  {\alpha,\beta,\gamma \in \mathbb{N}}$ and
  \begin{align}
    \Lambda_k(\Sph^2; \G) &= \overline{\lambda}_k(g_{\Sph^2}) &\text{ when } &1\leq k \leq 11,
    \\ \Lambda_k(\Sph^2; \G) &= 8\pi k &\text{ when } &k \in \mathbb{N}_\G \cup [60,\infty).
  \end{align}
\end{theorem}
\begin{remark}
  As in the case of $\mathbb{D}^2$, the configuration of $k' + 1$
  spheres might not be unique. For example, if $\G \in \brc{T_h, O, O_h}$,
  then $\Lambda_{10}$ and $\Lambda_{11}$ have exactly two maximizing
  configurations consisting either of $9$ or $7$ spheres,
  \begin{gather}
    \Lambda_{10}(\Sph^2; \G) = 8\,\overline{\lambda}_1(g_{\Sph^2})+\overline{\lambda}_2(g_{\Sph^2})
  = 6\,\overline{\lambda}_1(g_{\Sph^2})+ \overline{\lambda}_4(g_{\Sph^2}),
  \\
  \Lambda_{11}(\Sph^2; \G) = 8\,\overline{\lambda}_1(g_{\Sph^2})+\overline{\lambda}_3(g_{\Sph^2})
  = 6\,\overline{\lambda}_1(g_{\Sph^2})+ \overline{\lambda}_5(g_{\Sph^2}).
  \end{gather}
\end{remark}

\subsection{Ideas of the proof of existence and regularity}
The proof of the existence is a simplified version of the one presented
in~\cite{Karpukhin-Nadirashvili-Penskoi-Polterovich:2022:existence}
and basically consist of the following three steps.
\begin{enumerate}
  \item Firstly, we make use of
  the Clarke subdifferential calculus, to obtain a maximizing
  sequence similar to that of~\cite{Karpukhin-Nadirashvili-Penskoi-Polterovich:2022:existence}.
  The maximizing sequence generates maps of eigenfunctions
  $\Phi_\epsilon\colon\Sigma\to\R^n$ to some
  $\R^n$, where $n$ does not depend on the sequence
  because of multiplicity bounds in $\dim \Sigma = 2$.
  \item Then there exists a finite set of points where
  the limiting measure may have atoms, potentially leading to bubbling.
  On its complement,
  we provide a simple proof of the $H^1$-convergence (see Lemma~\ref{lem:strong-conv})
  of the corresponding eigenmaps.
  The limiting map
  is shown to be weakly (free boundary) harmonic, which, in dimension $2$,
  automatically implies smoothness.
  \item Finally, we deal with the bubbles.
  In Appendix~\ref{app:blow-up}, it is shown that the atoms
  of the limiting measure can be blown up into
  some continuous measures on $\R^2$. In dimension two
  that allows us to lift the maximizing sequence
  onto several copies of $\Sph^2$ or $\mathbb{D}^2$. The
  regularity of the blown-up measures can be derived from
  step 2 above. Once we know that the measures are regular
  enough, the limiting spectrum
  is estimated through the spectrum of the disjoint union,
  and $\Lambda_k(\Sph^2)$ and $\mathfrak{S}_k(\mathbb{D}^2)$
  are known.
\end{enumerate}

\subsection{Acknowledgements}
This paper is part of the author's PhD project under the supervision
of Mikhail Karpukhin and Iosif Polterovich. The author
thanks Mikhail Karpukhin for suggesting the problem
and Daniel Stern for sharing
his notes on the regularity and existence results,
which were useful in proving Lemma~\ref{lem:strong-conv}.
The author is partially supported by the ISM Scholarship.

\section{Preliminaries}\label{sec:prelim}

\subsection{Eigenvalues of measures}

The measure-theoretic approach to the eigenvalues firstly
appeared in~\cite{Grigoryan-Netrusov-Yau:2004:eignval-of-ellipt-oper} and was
developed further in~\cite*{Kokarev:2014:measure-eigenval, Girouard-Karpukhin-Lagace:2021:continuity-of-eigenval} in the context
of eigenvalue optimization.
Following~\cite{Kokarev:2014:measure-eigenval}, given a
positive Radon measure $\mu$ on a compact surface $\Sigma$,
one defines the variational eigenvalues $\lambda_k(\mu) \in [0,\infty]$,
$k\in\mathbb{N}$, of that measure by the formula
\begin{equation}\label{eq:variat-character}
  \lambda_k(\mu) = \inf_{F_{k+1}} \sup_{\phi \in F_{k+1}\setminus\brc{0}}
    \frac{\int\abs{d\phi}^2_g dv_g}{\int\phi^2 d\mu},
\end{equation}
where $F_{k+1} \subset C^\infty(\Sigma)$ runs over all $(k+1)$-dimensional
subspaces. If $\mu = \dv{g}$ is the volume measure of the metric $g$,
we obtain the classical eigenvalues $\lambda_k(g)$ of
the Laplacian $\Delta_g$.
If $\mu = \ds{g}$ is the boundary measure, we obtain the Steklov
eigenvalues $\sigma_k(g)$.

Let $H^{k,p}(\Sigma)$, also denoted by $W^{k,p}(\Sigma)$, be the corresponding Sobolev spaces on $\Sigma$,
and $H^k := H^{k,2}$.
If the measure $\mu$ is sufficiently regular, e.g.,
$\mu \in H^{1,p}(\Sigma)^*$ for some $1\leq p < 2$ (see~\cite{Girouard-Karpukhin-Lagace:2021:continuity-of-eigenval}), then
$\lambda_k(\mu)$ is an eigenvalue of the problem
\begin{equation}\label{eq:eigenval-of-meas}
  \Delta \phi = \lambda \phi \mu
\end{equation}
in the weak sense, i.e. for any test function $\psi \in C^\infty(\Sigma)$, one has
\begin{equation}\label{eq:weak-solution}
  \int \brt{d\phi, d\psi} dv_g = \lambda\int \phi \psi d\mu.
\end{equation}

Let $\rho^2 g \in [g]$ be another metric from the same conformal
class. Since $\dim \Sigma = 2$, the expression $|d\phi|^2_g dv_g$
does not depend on the particular choice of the metric, and the
volume forms will change accordingly, $\dv{\rho^2 g} = \rho^2 \dv{g}$
and $\ds{\rho^2 g} = \rho\ds{g}$. Thus,
$\lambda_k(\mu) = \lambda_k([g],\mu)$ depends only on the conformal
class of the metric $g$. We will omit the conformal class parameter
when it doesn't cause any confusion. From Proposition~1.1 and Theorem~$A_k$ in~\cite{Kokarev:2014:measure-eigenval},
one has
\begin{proposition}\label{prop:upper-cont-and-bound}
  The functional $\mu \mapsto \lambda_k(\mu)$ is weakly$^*$ upper semi-continuous
  on the space of (positive) Radon measures $\mathcal{M}^+(\Sigma)$, i.e.
  \begin{equation}
    \limsup_n \lambda_k(\mu_n) \leq \lambda_k(\mu),
    \quad\text{where } \mu_n \overset{w^*}{\to} \mu \in \mathcal{M}^+(\Sigma).
  \end{equation}
  The functional  $\mu \mapsto \overline{\lambda}_k(\mu) := \lambda_k(\mu) \mu(\Sigma)$
  is bounded on the space of \emph{continuous} Radon measures $\mathcal{M}^+_c(\Sigma)$.
  Namely, there exists a constant $C = C([g])$ such that
  \begin{equation}
    \overline{\lambda}_k(\mu) = \overline{\lambda}_k([g],\mu)
    \leq Ck, \quad \mu \in \mathcal{M}^+_c(\Sigma).
  \end{equation}
\end{proposition}

It turns out
that $\Lambda_k(\Sigma, [g])$ and $\mathfrak{S}_k(\Sigma, [g])$
can be thought of as the suprema of $\overline{\lambda}_k(\mu)$ over some
subsets of measures $\mathcal{C}$ and $\mathcal{C}^\partial$ to be
specified later,
\begin{equation}
  \Lambda_k(\Sigma, [g]) = \sup_{\mu \in \mathcal{C}}\overline{\lambda}_k([g],\mu),
  \quad
  \mathfrak{S}_k(\Sigma, [g]) = \sup_{\mu \in \mathcal{C}^\partial}\overline{\lambda}_k([g],\mu).
\end{equation}
It is clear
from Lemma~\ref{eq:variat-character-reversed} (see also~\cite*{Girouard-Karpukhin-Lagace:2021:continuity-of-eigenval})
that one can choose $\mathcal{C} = L^\infty(\Sigma)$ and
$\mathcal{C}^\partial = L^\infty(\partial\Sigma)$. Now, let $\mu = \rho \dv{g}$ or
$\mu = \rho \ds{g}$, where $\rho \in L^1$. If we set
$\mu_n := \min\brc{\mu, n}$, it is clear from the variational
characterization that
\begin{equation}
  \lambda_k(\mu) \leq \lambda_k(\mu_n).
\end{equation}
Moreover, $\mu_n(\Sigma) \to \mu(\Sigma)$, so $\overline{\lambda}_k(\mu)
\leq \liminf_{n} \overline{\lambda}_k(\mu_n)$. Since
$L^\infty \subset L^1$, we obtain that
\begin{equation}\label{eq:L1-densities}
  \Lambda_k(\Sigma,[g]) = \sup_{\mu \in L^1(\Sigma)}
  \overline{\lambda}_k([g],\mu),
  \quad
  \mathfrak{S}_k(\Sigma,[g]) = \sup_{\mu \in L^1(\partial\Sigma)}
  \overline{\lambda}_k([g],\mu).
\end{equation}

Let a discrete group $\G$ act by conformal transformations on $\brr{\Sigma, [g]}$.
We choose the reference metric $g$ to be invariant under
the action of $\Gamma$. In other words, $\G$ begins acting
by isometries on $\brr{\Sigma, g}$.
Denote by $\mathcal{P}(\Sigma) \subset \mathcal{M}^+(\Sigma)$ the space of all
Radon probability measures and define
\begin{gather}
  \label{eq:Lambda-class}
  \mathcal{C}_\G :=
  \set{\mu \in L^\infty(\Sigma)\cap \mathcal{P}(\Sigma) }{
  \gamma_* \mu = \mu, \ \forall\gamma\in\G},
  \\ \label{eq:Steklov-class}
  \mathcal{C}^\partial_\G :=
  \set{\mu \in L^\infty(\partial\Sigma)\cap\mathcal{P}(\Sigma)}
  {\gamma_* \mu = \mu, \ \forall\gamma\in\G}.
\end{gather}
From the discussion above, the following suprema corresponds
to the symmetric version of $\Lambda_k$ and $\mathfrak{S}_k$,
\begin{equation}\label{eq:meas-def}
  \Lambda_k(\Sigma, [g];\G)
  = \sup_{\mu \in \mathcal{C}_\G}\lambda_k([g], \mu),
  \quad
  \mathfrak{S}_k(\Sigma,[g];\G)
  = \sup_{\mu \in \mathcal{C}_\G^\partial}\lambda_k([g], \mu).
\end{equation}
Note that we could have chosen $L^1$ measures instead of $L^\infty$,
but they are less convenient in practice, since the
spectrum may not be discrete.

\subsection{Clarke subdifferentials}
Let $E$ be a Banach space, $U \subset E$ be an open subset,
and $f\colon U \to \mathbb{R}$ be a locally Lipschitz function.
One defines \emph{Clarke directional derivative} of $f$ at
$x \in U$ in the direction $v \in E$ as follows,
\begin{equation}
 f^{\circ}(x, v) = \limsup_{\tilde{x} \to x, \tau\downarrow 0}
 \frac{1}{\tau} \brs{f(\tilde{x} + \tau v) - f(\tilde{x})}.
\end{equation}
Since $f^{\circ}(x, v)$ is a sublinear, Lipschitz-continuous function of $v$, it is
a support functional of a convex closed subset of $E^*$, the dual of $E$.
The Clarke subdifferential of $f$ at the point $x$ is defined as
\begin{equation}
  \partial_{\circ} f(x) = \set*{\xi \in E^*}{\brt{\xi, v}
  \leq f^{\circ}(x, v) \ \forall v \in E}.
\end{equation}
The following properties of
$\partial_{\circ} f(x)$ can be found, e.g., in.~\cite[Section~7.3]{Schirotzek:2007:nonsmooth-analysis}:
\begin{itemize}
  \item $\partial_{\circ} f(x) \neq \varnothing$;
  \item $\partial_{\circ} f(x)$ is closed, convex, and bounded set hence weakly$^*$ compact;
  \item $0 \in \partial_{\circ} f(x)$ if $x$ is a local maximum/minimum.
\end{itemize}
The Clarke subdifferential
generalizes the notion
of the derivative in the case of locally Lipschitz functions
and captures many useful properties of the derivative that
are used in optimization problems.
Different applications of the Clarke subdifferential to critical metrics
and eigenvalue optimization problems can be found in~\cite{Petrides-Tewodrose:2024:eigenvalue-via-clarke}.

\begin{proposition}
  Let $g \colon F \to E$ be a continuously Fréchet-differentiable map of Banach spaces defined
  on a  neighborhood of $x\in F$. If
  $f\colon E \to \mathbb{R}$ is a Lipschitz function
  on a neighborhood of $g(x) \in E$, one has
  \begin{equation}
    \partial_{\circ} (g^*f)(x) \subset \set*{(d_x g)^*\xi}{\xi\in \partial_{\circ} f(g(x))}.
  \end{equation}
  Thus, the subdifferential of a pullback is contained in the pullback of the subdifferential.
\end{proposition}
\begin{proof}
  Since
  \begin{equation}
    \frac{1}{\tau}\norm{g(\tilde{x} + \tau v) - g(\tilde{x}) - d_{x} g(\tau v)}
    \leq \sup_{0\leq \theta \leq \tau} \norm{d_{\tilde{x} + \theta v} g(v) - d_{x} g(v)}
  \end{equation}
  and $d_{\tilde{x} + \theta v} g(v) \to d_x g(v)$ as $\tilde{x} \to x$, $\tau\downarrow 0$,
  one easily concludes that
  \begin{align}
    (g^*f)^{\circ}(x, v) &= \limsup_{\tilde{x} \to x, \tau\downarrow 0}
    \frac{1}{\tau} \brs{f(g(\tilde{x} + \tau v)) - f(g(\tilde{x}))}
    \\ &\leq \limsup_{\tilde{x} \to x, \tau\downarrow 0}
    \frac{1}{\tau} \brs{f(g(\tilde{x} + \tau v)) - f(g(\tilde{x})+ d_{x} g(\tau v))}
    \\ &+ \limsup_{\tilde{x} \to x, \tau\downarrow 0}
    \frac{1}{\tau} \brs{f(g(\tilde{x})+ \tau d_{x}g(v)) - f(g(\tilde{x}))}
    \\ & \leq f^{\circ}(g(x), d_x g(v)).
  \end{align}
  Then the linear map $d_x g(v) \mapsto \brt{\zeta,v}$,
  $\zeta \in \partial_{\circ} (g^*f)(x)$, is well defined and continuous
  since
  \begin{equation}
    \pm\brt{\zeta,v} \leq f^{\circ}(g(x), \pm d_x g(v)) \leq C \norm{g(v)}.
  \end{equation}
  By Hahn–Banach theorem, it can be extended to a map $\xi \in E^*$ such that
  $\brt{\xi,w} \leq f^{\circ}(g(x), w)$, $w\in E$. Thus, $\xi \in \partial_{\circ} f(g(x))$
  and $(d_x g)^*\xi = \zeta$.
\end{proof}
The following proposition is a simplified version of~\cite[Theorem~12.4.1]{Schirotzek:2007:nonsmooth-analysis}
sufficient for our purposes.
\begin{proposition}[Clarke’s Multiplier Rule]\label{prop:multip-rule}
  Let $f\colon U \to \mathbb{R}$ be a locally Lipschitz function defined
  on an open set $U$ of a Banach space $E$. Suppose that $U$ contains
  a closed convex set $C$, $x\in C$, and $f(x) = \min_{\tilde{x} \in C} f(\tilde{x})$.
  Then
  \begin{equation}
    \exists \xi \in \partial_{\circ} f(x)\colon
    \brt{\xi, \tilde{x}-x} \geq 0 \ \forall \tilde{x} \in C.
  \end{equation}
\end{proposition}

\subsection{Measures as bilinear forms}\label{subsec:meas-bilin}
Let us recall the weak eigenfunction equation~\eqref{eq:weak-solution}.
So it is natural to expect the bilinear form
$\brr{\phi, \psi} \mapsto \int \phi \psi d\mu$ to be
continuous as a form $H^1(\Sigma) \times H^1(\Sigma) \to \R$.
If it really is, one can reformulate
equation~\eqref{eq:weak-solution} and
variational characterization~\eqref{eq:variat-character}
in a more convenient form. It is clear that
$\lambda_0(\mu)$ is always $0$ and the corresponding
eigenfunctions are just constants provided $\Sigma$ is connected.
Let us denote by $X$ the Hilbert space
\begin{equation}
  X := \set*{\phi \in H^1}{\int \phi dv_g = 0},
  \quad
  \brr{\phi, \psi}_X := \int \brt{d\phi, d\psi} dv_g,
\end{equation}
and by $[\phi]_\mu$ the average
\begin{equation}
  [\phi]_\mu := \frac{1}{\mu(\Sigma)}\int \phi d\mu.
\end{equation}
For $k \geq 1$, we see that
\begin{equation}\label{eq:variat-character-reversed}
  \lambda_k^{-1}(\mu) = \sup_{
    \substack{F \subset X
    \\ \dim F = k}}
    \inf_{\phi \in F}
    \frac{\int \brr{\phi - [\phi]_\mu}^2 d\mu}{\int\abs{d\phi}^2 dv_g},
\end{equation}
and equation~\eqref{eq:weak-solution} is equivalent to
\begin{equation}
  \lambda^{-1} \brr{\phi, \psi}_X = Q(\mu)\brs{\phi, \psi}
\end{equation}
for any $\psi \in X$, where
\begin{equation}
  Q(\mu)\brs{\phi, \psi} := \int \brr{\phi - [\phi]_\mu}\brr{\psi - [\psi]_\mu} d\mu.
\end{equation}
Moreover, $\phi \in X$ is an eigenfunction corresponding to $\lambda_k^{-1}(\mu)$
if and only if $\phi - [\phi]_{\mu} \in H^1$ is an eigenfunction corresponding
to $\lambda_k(\mu)$.

Slightly abusing the notation in the next lemma, we write $\phi \mapsto \int \phi d\mu$
for functionals even if they are not induced by a measure.
\begin{lemma}
  Let $1\leq p < 2$ and $\mu \in H^{1,p}(\Sigma)^*$ such that
  $\brt{\mu, 1} \neq 0$. Then the map $\mu \mapsto Q(\mu)$ is
  continuously Fréchet differentiable as a map from
  a neighborhood $U \subset (H^{1,p})^*$ of $\mu$ into
  the space of continuous bilinear forms,
  $U \to \mathcal{B}[X]$. Its derivative
  at $\mu$ is
  \begin{equation}
    \nu \mapsto \int \brr{\phi - [\phi]_\mu}\brr{\psi - [\psi]_\mu} d \nu,
  \end{equation}
  where $\nu \in (H^{1,p})^*$.
  The bounded linear operator $T(\mu) \in \mathcal{L}[X]$
  induced by $Q(\mu)$ is self-adjoint and compact.
\end{lemma}
\begin{proof}
  Notice that the quadratic form
  \begin{equation}
    \int \brr{\phi - [\phi]_\mu}^2 d\mu = \int \phi^2 d\mu -
    \brr{\frac{1}{\brt{\mu, 1}}\int \phi d\mu}^2
  \end{equation}
  depends rationally on the expressions of the form
  $\int \phi\psi d\mu$, where $\phi,\psi \in H^1$. So, it is enough
  to prove the continuity of $\mu \mapsto \int \phi\psi d\mu$,
  which follows from the estimate, see~\cite[Lemma~4.9]{Girouard-Karpukhin-Lagace:2021:continuity-of-eigenval},
  \begin{equation}
    \abs{\int \phi\psi d\mu} \leq
    c_1 \norm{\phi}_{H^1} \norm{\psi}_{H^1} \norm{\mu}_{(H^{1,p})^*}
  \end{equation}
  and the fact that $H^1$-norm is an equivalent
  norm on $X$ by the Poincaré inequality.

  The compactness of $T(\mu)$ follows from the facts that the
  images of $L^\infty$
  measures are dense in $(H^{1,p})^*$, and for
  $\mu = \rho \dv{g} \in L^\infty$, one has
  $T(\mu) = i^* \tilde{T}(\mu) i$, where
  $i\colon X \hookrightarrow  L^2$ is
  compact.
\end{proof}
Let $E_k(\mu) \subset H^1(\Sigma)$ be
the space of eigenfunctions corresponding
to $\lambda_k(\mu)$.
\begin{lemma}\label{lem:subdiff-calc}
  Let $1\leq p < 2$ and $\mathcal{P} = \mathcal{P}(\Sigma)$
  be the space of probability measures. Then
  the functional $\lambda_k^{-1} \colon \mathcal{P}
  \cap (H^{1,p})^* \to \R$ has a locally
  Lipschitz extension on a neighborhood of
  $\mathcal{P} \cap (H^{1,p})^* \subset (H^{1,p})^*$.
  Its Clarke subdifferential can be described by the inclusion
  \begin{equation}
    \partial_\circ \lambda_k^{-1}(\mu) \subset co\
    \set*{\phi^2 }{
      \phi \in E_k(\mu),\ \int \abs{d\phi}^2 dv_g = 1
    },
  \end{equation}
  where $\text{co}\,K$ denotes the convex hull of the set $K$.
\end{lemma}
\begin{proof}
  Let us fix $\mu$ and take a look at $\lambda_k^{-1}(\mu)$.
  It is an isolated point of the spectrum
  $Spec\,[T(\mu)]$
  with the multiplicity $m = \dim E'_k(\mu)$,
  where $E'_k(\mu) \subset X$ is the space of eigenfunctions
  corresponding to $\lambda_k^{-1}(\mu)$.
  By the previous lemma,
  the map $\tilde{\mu} \mapsto T(\tilde{\mu})$
  is of class $C^1$. Hence, the holomorphic
  functional calculus gives us a neighborhood $W$
  of $\mu$ and the $C^1(W, X)$ family
  of orthogonal projectors $\tilde{\mu} \mapsto P(\tilde{\mu})$
  of rank $m$ onto the space of eigenvalues of
  $T(\tilde{\mu})$ that are close to $\lambda_k^{-1}(\mu)$
  so that $E'_k(\tilde{\mu}) \subset im\,P(\tilde{\mu})$

  Fix an orthonormal basis $\brc{e}_{i=1}^m$ of $E'_k(\mu)$.
  After possibly shrinking the neighborhood $W$,
  we may suppose that $\brc{P(\tilde{\mu})e_i}_{i=1}^m$ is
  still a basis of $im\,P(\tilde{\mu})$ for any
  $\tilde{\mu} \in W$. Then one can apply
  the Gram-Schmidt process to obtain a $C^1$-family
  of orthogonal transformations
  $U(\tilde{\mu})\colon E_k(\mu) \to im\,P(\tilde{\mu})$ so that
  \begin{equation}
    \lambda_k^{-1}(\tilde{\mu}) \in Spec\,[T_0(\tilde{\mu})],
  \end{equation}
  where $T_{0}(\tilde{\mu}) :=
  U(\tilde{\mu})^*T(\tilde{\mu})U(\tilde{\mu})$ is
  a self-adjoint operator on a finite-dimensional
  space, $T_{0}(\tilde{\mu}) \in \text{End}[E'_k(\mu)]$.

  Let us compute the Clarke directional derivative
  $(\lambda_k^{-1})^\circ(\mu, \nu).$ By a
  theorem of Lidskii, see~\cite[Theorem~III.6.10]{Kato:1995:perturbation-theory}, the difference
  of $\lambda_i(B) - \lambda_i(A)$ eigenvalues lies in the convex hull
  of $Spec\,(B-A)$. In particular, one has
  \begin{equation}
    \tau^{-1}\brs{\lambda_k^{-1}(\tilde{\mu} + \tau \nu) -
    \lambda_k^{-1}(\tilde{\mu})} \leq \sup_{\substack{
      \phi \in E'_k(\mu) \\ \norm{\phi} = 1
    }} \brr{\tau^{-1}[T_{0}(\tilde{\mu} + \tau \nu) -
    T_{0}(\tilde{\mu})]\phi, \phi}_X.
  \end{equation}
  Thus,
  \begin{equation}
    \begin{aligned}
      (\lambda_k^{-1})^\circ(\mu, \nu) &=
      \limsup_{\tilde{\mu}\to \mu,\, \tau\downarrow 0}
      \tau^{-1}\brs{\lambda_k^{-1}(\tilde{\mu} + \tau \nu) -
      \lambda_k^{-1}(\tilde{\mu})}
      \\ &\leq \sup_{\substack{
        \phi \in E'_k(\mu) \\ \norm{\phi} = 1
      }} \frac{d}{d\tau}\biggr|_{\tau = 0} \brr{T_{0}(\mu + \tau \nu)\phi, \phi}_X,
    \end{aligned}
  \end{equation}
  and
  \begin{equation}
   \begin{aligned}
    \frac{d}{d\tau}\biggr|_{\tau = 0}\brr{T_{0}(\mu + \tau \nu)\phi, \phi}_X
    &=\frac{d}{d\tau}\biggr|_{\tau = 0} Q(\mu + \tau \nu)
    \brs{U(\mu + \tau \nu)\phi, U(\mu + \tau \nu)\phi}
    \\ &=\frac{d}{d\tau}\biggr|_{\tau = 0} Q(\mu + \tau \nu)
    \brs{\phi, \phi}
    = \int \brr{\phi - [\phi]_\mu}^2 d \nu
   \end{aligned}
  \end{equation}
  since $\frac{d}{d\tau}|_{\tau = 0}
  U(\mu + \tau \nu)\phi \ \bot\ \phi$.
  After replacing $\phi$ by
  $\tilde{\phi} := \phi - [\phi]_\mu \in E_k(\mu)$,
  we see that
  \begin{equation}
    (\lambda_k^{-1})^\circ(\mu, \nu) \leq \sup_{\psi \in K} \brt{\psi, \nu}
  \end{equation}
  where
  \begin{equation}
    K := \set*{\phi^2 }
    {\phi \in E_k(\mu),\ \int \abs{d\phi}^2 dv_g = 1}.
  \end{equation}
  it is clear from the definition of the Clarke subdifferential
  that
  \begin{equation}
    \partial_\circ \lambda_k^{-1}(\mu) \subset
    \set{\xi \in H^{1,p}}{\forall \nu \in (H^{1,p})^*\colon \brt{\xi, \nu} \le
    \sup_{\psi \in K} \brt{\psi, \nu}},
  \end{equation}
  and it can happen if and only if $\xi \in co\, K$. In fact,
  $\nu \mapsto \sup_{\psi \in K} \brt{\psi, \nu}$ is the support
  functional of the set $co\, K$.
\end{proof}

\section{Proof of existence and regularity}\label{sec:exis-and-reg}

We adopt the convention that $c, c_1, c_2, \cdots$ are
some positive constants. The dependencies
will be specified if necessary. By $\epsilon \in \brc{\epsilon_i}$ we
denote a sequence of positive numbers, and $o(1)$ denotes
any sequence such that
$o(1) \to 0$ as $\epsilon \to 0$.

\subsection{Constructing a maximizing sequence}

Recall the definitions~\eqref{eq:Lambda-class} of $\mathcal{C}_{\G}$
and~\eqref{eq:Steklov-class} of $\mathcal{C}_{\G}^\partial$.
Let $\mathcal{C}^\star_{\G} \in \brc{\mathcal{C}_{\G}, \mathcal{C}_{\G}^\partial}$
be either of them.
By the trace embedding, $\mathcal{C}^\star_{\G} \subset (H^{1,1})^*$, and therefore,
one can use the theory presented in Section~\ref{subsec:meas-bilin}.
\begin{proposition}\label{prop:max-sequence}
  Let $\lambda := \sup_{\mu\in \mathcal{C}^\star_{\G}} \lambda_k(\mu)$ and
  $(\Sigma,[g])$ be a compact connected surface with a conformal
  action of a discrete group $\G$. Then there
  exist a number $n \in \mathbb{N}$, a sequence of
  probability measures $\mu_\epsilon \in \mathcal{C}^\star_{\G}$ with eigenvalues
  $\lambda_\epsilon:= \lambda_k(\mu_\epsilon) \to \lambda$
  as $\epsilon \to 0$, and a
  sequence of $\G$-equivariant maps $\Phi_\epsilon\in
  H^1(\Sigma, \R^n)$, i.e. $\gamma^* \Phi = r(\gamma) \Phi$
  for some (right) representation $r\colon \G \to O(\R^n)$,
  such that the following conditions are satisfied:
  \begin{enumerate}
    \item $\Delta \Phi_\epsilon = \lambda_\epsilon \Phi_\epsilon
    \mu_\epsilon$ weakly;
    \item $\Phi_\epsilon\colon \Sigma \to \mathbb{B}^n$, i.e. $\abs{\Phi_\epsilon}\leq 1$,
    and $\norm{d\Phi_\epsilon}_{L^2(\Sigma)} \leq C$ uniformly;
    \item\label{it:almost-one} $\epsilon \norm{\mu_\epsilon}_{L^\infty} \leq 1 $ and either
    \begin{itemize}
      \item $v_g\brr{\brc{x \in \Sigma \colon \abs{\Phi_\epsilon(x)} < 1}} \leq \epsilon$
      if $\partial \Sigma = \varnothing$, or
      \item $s_g\brr{\brc{x \in \partial\Sigma\colon\abs{\Phi_\epsilon(x)} < 1}} \leq \epsilon$
      if $\partial \Sigma \neq \varnothing$.
    \end{itemize}
  \end{enumerate}
\end{proposition}
\begin{proof}
  In the corresponding $L^\infty$ space, consider the ball
  $B(\epsilon^{-1})=B(\epsilon^{-1};0) \subset L^\infty$ of
  radius $\epsilon^{-1}$ centered at zero.
  If
  \begin{equation}
    \lambda_\epsilon := \sup \set{\lambda_k(\mu)}
    {\mu \in B(\epsilon^{-1};0)\cap \mathcal{C}^\star_{\G}},
  \end{equation}
  it is clear that $\lambda_\epsilon \to \lambda$ as $\epsilon \to 0$.
  At the same time, $B(\epsilon^{-1};0)\cap \mathcal{C}^\star_{\G}$ is a bounded
  weakly$^*$ closed subset of $L^\infty$, hence weakly$^*$ compact.
  Proposition~\ref{prop:upper-cont-and-bound} implies that
  the functional $\mu \to \lambda_k(\mu)$ is bounded on $\mathcal{C}^\star_{\G}$
  and upper semi-continuous under the weak$^*$ convergence. So,
  one easily finds an element $\mu_\epsilon \in \mathcal{C}^\star_{\G}$ realizing
  $\lambda_\epsilon = \lambda_k(\mu_\epsilon)$.

  The next step is to apply Lemma~\ref{lem:subdiff-calc} together
  with Proposition~\ref{prop:multip-rule},
  to obtain $\psi \in \partial_\circ \lambda_k^{-1}(\mu_\epsilon)$
  such that
  \begin{equation}
    \int \psi d(\mu - \mu_\epsilon) \geq 0
  \end{equation}
  for all $\mu \in B(\epsilon^{-1})\cap \mathcal{C}^\star_{\G}$.
  The expression for $\psi \in \partial_\circ \lambda_k^{-1}(\mu_\epsilon)$
  takes the following form
  \begin{equation}
    \psi = \sum_{i=1}^{n_\epsilon} \phi_i^2,
  \end{equation}
  where $\sum \int \abs{d\phi_i}^2 dv_g = 1$,
  and $n_\epsilon$ can be chosen so that $n_\epsilon\leq n = n(\Sigma,k)$, in view of the
  multiplicity bounds established in~\cite{Kokarev:2014:multiplicity-bounds}
  for the closed case ($\partial \Sigma = \varnothing$),
  and in~\cite{Karpukhin-Kokarev-Polterovich:2014:steklov-multiplicity-bounds} (or, for more regular densities, \cite{Jammes:2014:steklov-mult,Fraser-Schoen:2016:sharp-steklov-bounds})
  for the Steklov problem.

  Therefore, we can set
  \begin{equation}
    \tilde{\Phi}_\epsilon := (\phi_1, \cdots, \phi_{n})
    \quad and \quad
    \Phi_\epsilon := \frac{1}{\sqrt{\abs{\G}}}\sum_{\gamma\in\G}
    e_{\gamma}\otimes \gamma^*\tilde{\Phi}_\epsilon \in \R^{\abs{\G}} \otimes \R^n
  \end{equation}
  to see that $\abs{\gamma^* \Phi_\epsilon}^2 = \abs{\Phi_\epsilon}^2$, $\Phi_\epsilon$ is $\G$-equivariant,
  \begin{equation}\label{ineq:extremal-condition}
    \int \abs{\Phi_\epsilon}^2 d\mu_\epsilon
    \leq \int \abs{\Phi_\epsilon}^2 d\mu
  \end{equation}
  for any $\mu \in B(\epsilon^{-1})\cap \mathcal{C}^\star_{\G}$,
  and
  \begin{equation}\label{ineq:sup-norm-lower-bound}
    1 = \norm{d\Phi_\epsilon}_{L^2} =
  \lambda_\epsilon \int \abs{\Phi_\epsilon}^2 d\mu_\epsilon
  \leq \lambda_\epsilon \norm{\Phi_\epsilon}_{L^\infty}^2.
  \end{equation}
  To finish the proof, let us deal with two cases separately.

  \textbf{The case }$\mathbf{\partial \Sigma = \varnothing}$.
  One has $\mu_\epsilon = \rho_\epsilon \dv{g}$,
  $\mu = \rho \dv{g}$ for some $\rho_\epsilon, \rho \in L^\infty(\Sigma)\cap B(\epsilon^{-1})$.
  Then, by elliptic regularity, $\Phi_\epsilon \in H^{2,p}(\Sigma)$
  for any $p < \infty$, so $\Phi_\epsilon$ is continuous.
  We would like to show that
  inequality~\eqref{ineq:extremal-condition} implies
  \begin{equation}\label{eq:const-density}
    \rho_\epsilon \equiv \epsilon^{-1}
    \text{ a.e. on }
    \brc{x\in\Sigma \colon \abs{\Phi_\epsilon(x)} < M},
  \end{equation}
  where $M:= \norm{\Phi_\epsilon}_{L^\infty} \geq \lambda^{-\frac{1}{2}}$
  by~\eqref{ineq:sup-norm-lower-bound}. If we did so,
  one could check that the renormalized maps
  $\Psi_\epsilon := \frac{1}{M}\Phi_\epsilon$
  satisfy the conditions of the proposition.

  If~\eqref{eq:const-density} does not hold, then
  for any small enough $\delta > 0$, one has
  $(\epsilon^{-1}-\rho_\epsilon)\chi_A \not\equiv 0$, where
  $\chi_A$ is the characteristic function of the set
  \begin{equation}
    A:= \brc{x\colon \abs{\Phi_\epsilon(x)} < M - 2\delta}.
  \end{equation}
  Since $\Phi_\epsilon$ is continuous, the set
  \begin{equation}
    B := \brc{x\colon M - \delta < \abs{\Phi_\epsilon(x)}}
  \end{equation}
  is open, and $\rho_\epsilon \chi_B \neq 0$ as well,
  or else we would apply the maximum
  principle to the subharmonic function
  $\abs{\Phi_\epsilon}^2$ on $B$ and conclude that $\abs{\Phi_\epsilon} \equiv M$
  on $\Sigma$, which is connected. That allows us
  to take
  \begin{equation}
    \rho = \rho_\epsilon + t\brr{
      \frac{(\epsilon^{-1}-\rho_\epsilon)\chi_A}
        {\norm{(\epsilon^{-1}-\rho_\epsilon)\chi_A}}_{L^1}
      - \frac{\rho_\epsilon \chi_B}
        {\norm{\rho_\epsilon \chi_B}}_{L^1}
    } \in B(\epsilon^{-1})\cap \mathcal{C}^\star_{\G}
  \end{equation}
  for small enough $t$ ($\rho$ is $\G$-invariant since
  $\abs{\gamma^* \Phi_\epsilon} = \abs{\Phi_\epsilon}$).
  After plugging in $\mu = \rho \dv{g}$ into
  inequality~\eqref{ineq:extremal-condition}, one obtains
  \begin{equation}
    M-\delta \leq M - 2\delta,
  \end{equation}
  which is a contradiction.

  \textbf{The case }$\mathbf{\partial \Sigma \neq \varnothing}$.
  The idea is the same, but now, we are working on
  $\partial\Sigma$. One has
  $\mu_\epsilon = \rho_\epsilon \ds{g}$,
  $\mu = \rho \ds{g}$ for some $\rho_\epsilon, \rho \in
  L^\infty(\partial\Sigma)\cap B(\epsilon^{-1})$.
  Again, $\Phi$ is continuous
  on the boundary because the Dirichlet-to-Neumann operator
  is an elliptic pseudodifferential operator of order 1 on $\partial \Sigma$,
  so by $H^{k,p}$-regularity (see, e.g.,~\cite[Proposition 13.6.5]{Taylor:2023:pde-3}), $\Phi \in H^{1,p}(\partial\Sigma) \subset C(\partial\Sigma)$
  for any $p < \infty$.
  We would like to show that
  inequality~\eqref{ineq:extremal-condition} implies
  \begin{equation}
    \rho_\epsilon \equiv \epsilon^{-1}
    \text{ a.e. on }
    \brc{x \in \partial\Sigma \colon \abs{\Phi_\epsilon(x)} < M},
  \end{equation}
  and then $\frac{1}{M}\Phi_\epsilon$ would satisfy the conditions
  ($\frac{1}{M}\Phi_\epsilon< 1$ on $\Int \Sigma$
  by the maximum principle unless $\abs{\Phi}^2$ is a constant). The remaining step
  is to prove that
  $\rho_\epsilon \chi_B \neq 0$. If it is not the case,
  the continuity of $\Phi_\epsilon$ on the boundary
  implies $B \subset \partial\Sigma$ is open. One also has
  $\partial_\nu \Phi = 0$ on $B$, where $\partial_\nu$ is the
  outward-pointing normal to $\partial\Sigma$. But
  this contradicts with
  Zaremba's principle, which states that either
  $\abs{\Phi} \equiv M$, or
  $\partial_\nu \abs{\Phi}^2 > 0$
  at the points of maximum of $\abs{\Phi}^2$.
\end{proof}
\begin{remark}
  By analogous arguments involving Clarke subdifferentials and the maximum principle,
  one arrives at the same maximizing
  sequence $\brc{\Phi_\epsilon}$ through a different optimization problem.
  See~\cite{Karpukhin-Nadirashvili-Penskoi-Polterovich:2022:existence},
  one may consider
  $\int dV \to \max$
  over the set $\set*{V \in L^{\infty}}{V \geq 0, \ \nu_k(V) \geq 0}$,
  where $L^{\infty} \in \brc{L^{\infty}(\Sigma), L^{\infty}(\partial\Sigma)}$
  and $\nu_k(V)$ is the $k$-th eigenvalue of corresponding Schrödinger operator,
  \begin{equation}
    \Delta \phi - \phi V = \nu_k \phi
    \ \text{ on } L^2(\Sigma).
  \end{equation}
\end{remark}

\subsection{Convergence and regularity near stable points}
To treat the closed and Steklov cases simultaneously, we include $\partial\Sigma$ in the definition of $\Sigma$ and
consider a \emph{relatively} open set $\Omega \subset \Sigma$, meaning that
we allow $\Omega\cap\partial\Sigma$ to be nonempty.
Let us use the convention that the subscript $co$ means
\emph{relatively compactly supported}, e.g.
\begin{equation}
  H^1_{co}(\Omega) :=
\overline{\set{\phi \in C^\infty(\Sigma)}{\supp \phi \subset \Omega}}^{H^1}.
\end{equation}
That is,
$H^1_{co}(\Omega)$ coincides with $H^1_{0}(\Omega)$ if
$\Omega\cap\partial\Sigma = \varnothing$, and
$H^1_0(\Omega) \subsetneq  H^1_{co}(\Omega)$
if $\Omega\cap\partial\Sigma \neq \varnothing$.

Let $\mu$ be a Radon measure on $\Sigma$. Define
$\ind(\mu, \Omega) \in [0,\infty]$,
\begin{equation}
  \ind(\mu, \Omega) =  \max \brc{\dim V\colon V \subset C_{co}^\infty(\Omega)
  \text{ s.t. }Q_\mu|_V < 0},
\end{equation}
where
\begin{equation}
  Q_\mu(\phi) = \int \abs{d\phi}^2 \dv{g} - \int \phi^2 d\mu.
\end{equation}
Notice that $\supp \phi \cap \partial \Sigma$ can be nonempty
if $\Omega \cap \partial \Sigma \neq \varnothing$. If
$\Omega' \subset \Omega$, it is clear that
$\ind(\mu, \Omega') \leq \ind(\mu, \Omega)$.
Following the ideas from~\cite[Section~4.1]{Karpukhin-Nadirashvili-Penskoi-Polterovich:2022:existence},
see also~\cite[Section~3.2]{Karpukhin-Nahon-Polterovich-Stern:2021:stability},
let us introduce the following definition.
\begin{definition}
  A point $p\in \Sigma$ is said to be \emph{stable}
  for the sequence of measures
  $\brc{\mu_\epsilon}$
  if there exists a neighborhood $\Omega$ of $p$ such that up to a subsequence,
  \begin{equation}\label{good.pt}
    \ind(\mu_{\epsilon},\Omega) = 0
  \end{equation}
  The point $p$ is called \emph{unstable} otherwise.
\end{definition}

\begin{remark}\label{rem:size_of-bad}
If $\lambda_k(\mu_\epsilon) = 1$, there exist at most $k$
distinct \emph{unstable} points. Otherwise,
we could find $k+1$ disjoint open sets $\Omega_p\cap \Omega_q=\varnothing$ on which
$$
\ind(\mu_{\epsilon},\Omega_p) \geq 1.
$$
And the variational characterization of
$\lambda_k(\mu_\epsilon)$ would imply
that $\lambda_k(\mu_\epsilon) < 1$, which is
clearly a contradiction.
\end{remark}
\begin{remark}\label{rem:quadr-form-bound}
  It follows that for any $\phi\in C_{co}^{\infty}(\Omega)$
  in a neighborhood $\Omega$ of a stable point $p$, one has
  \begin{equation}\label{ineq:quadr-form-bound}
    \int \phi^2 d\mu_\epsilon \le
    \int \abs{d\phi}^2 dv_g.
  \end{equation}
  One can view
  $\mu_\epsilon \in \mathcal{B}[H^1_{co}(\Omega)]$ as uniformly bounded bilinear forms
  on $H^1_{co}(\Omega)$.
\end{remark}

Let us take a look at a sequence of eigenmaps $\Phi_\epsilon$
generated by Proposition~\ref{prop:max-sequence}. It is bounded
in $H^1$, so
passing to a subsequence, we have a $\G$-equivariant
map $\Phi: \Sigma\to \mathbb{B}^n$ and a $\G$-invariant Radon measure
$\mu$ such that $\mu_{\epsilon} \oset{w^*}{\to} \mu$
in $\mathcal{P}(\Sigma)$ and $\Phi_{\epsilon} \oset[0.2ex]{w}{\to}\Phi$
weakly in $H^{1}(\Sigma)$.
The \ref{it:almost-one}rd property of the proposition then implies
\begin{equation}\label{eq:limit-map-to-sphere}
  \begin{gathered}
    \abs{\Phi}^2 \equiv 1 \quad\text{if }
    \partial\Sigma = \varnothing\text{, or}
    \\\abs{\Phi}^2|_{\partial\Sigma} \equiv 1
      \quad\text{if } \partial\Sigma \neq \varnothing.
  \end{gathered}
\end{equation}
The proof of the next lemma is inspired by the notes of Daniel Stern~\cite{Stern:2024:private}.
Notice that the sequence $\brc{\Phi_\epsilon}$, as above, satisfy the conditions
of the lemma.
\begin{lemma}\label{lem:strong-conv}
  Let $\Sigma$ be a Riemannian manifold (possibly with boundary),
  ${\mu_\epsilon \overset{w^*}{\to} \mu}$ in
  $\mathcal{M}^+(\Sigma)$, and $\lambda_k(\mu_\epsilon)=1$.
  Denote by $D$ the (finite) collection of unstable points of
  $\mu_\epsilon$. Let
  $\Phi_\epsilon \oset[0.2ex]{w}{\to} \Phi$ in
  $H^1(\Sigma,\R^n)$
  such that
  \begin{enumerate}
    \item $\Delta\Phi_\epsilon = \Phi_\epsilon
    \mu_\epsilon$ weakly;
    \item\label{it:lower-than} $\begin{aligned}
      \limsup_{\epsilon\to0}\int
      (\abs{\Phi_\epsilon}^2 - \abs{\Phi}^2) \phi^2 d\mu_\epsilon \leq 0
      \quad\forall \phi \in C^\infty_{co}(\Sigma\setminus D).
    \end{aligned}$
  \end{enumerate}
  Then up to a subsequence, for any $\Omega \Subset \Sigma\setminus D$,
  \begin{equation}
    \Phi_\epsilon \to \Phi \text{ in } H^1(\Omega)
    \quad\text{and}\quad
    \mu_\epsilon \overset{w^*}{\to} \mu \text{ in }
    \mathcal{B}[H^1_{co}(\Omega)],
  \end{equation}
  i.e. $\int fg d\mu_\epsilon \to \int fg d\mu$ for any fixed pair
  $f,g \in H^1_{co}(\Omega)$.
\end{lemma}
\begin{proof}
  It is enough to show that there exists a
  neighborhood $W \Subset \Sigma\setminus D$ of a stable point $p$ such that the stated
  convergences take place.

  Let us choose a neighborhood $U \Subset \Sigma\setminus D$ and a subsequence $\brc{\mu_\epsilon}$
  from the definition of a stable point. Then
  Remark~\ref{rem:quadr-form-bound} implies
  that $\brc{\mu_\epsilon}$ is bounded in
  $\mathcal{B}[H^1_{co}(U)]$, hence precompact in the weak
  'operator' topology. By looking at the integrals
  $\int \phi^2 d\mu_\epsilon$ of smooth functions
  $\phi \in C^\infty_{co}(U)$, one sees that $\mu$ is the only
  accumulation point of $\brc{\mu_\epsilon}
  \subset\mathcal{B}[H^1_{co}]$.

  Fix a neighborhood $W \Subset U$ of $p$ and a function
  $\phi \in C^\infty_{co}(U)$ such that $\phi|_W \equiv 1$.
  Then $\Phi_{\epsilon}$ converges weakly to $\Phi$ in $H^{1}(U)$,
  and hence strongly in $L^2(U)$.
  For weakly convergent sequences, one has
  \begin{equation}
    \int_W \abs{d\Phi}^2\dv{g}
    \leq \liminf_{\epsilon\to0} \int_W \abs{d\Phi_\epsilon}^2\dv{g}.
  \end{equation}
  It remains show that
  \begin{equation}
    \limsup_{\epsilon\to0} \int \abs{d\Phi_\epsilon}^2\phi^2\dv{g}
    \leq \int \abs{d\Phi}^2\phi^2\dv{g}.
  \end{equation}
  That will imply $\lim_{\epsilon\to0}\norm{\Phi_\epsilon}_{H^1(W)} =
  \norm{\Phi}_{H^1(W)}$, and therefore,
  $\Phi_\epsilon \to \Phi$ in $H^1(W)$.

  Because of
  $L^2$ convergence, one has
  \begin{equation}
    \int \abs{d\phi}^2 \abs{\Phi_\epsilon - \Phi}^2 dv_g \to 0,
  \end{equation}
  and by the Cauchy–Schwarz inequality,
  \begin{gather}
    \abs{\int \brt{\phi d(\Phi_\epsilon - \Phi),
    (\Phi_\epsilon - \Phi) d\phi} dv_g}
    \leq \int \abs{\phi}\abs{d\Phi_\epsilon - d\Phi}
    \abs{\Phi_\epsilon - \Phi}\abs{d\phi}dv_g
    \\\to 0.
  \end{gather}
  since $\norm{d\Phi_\epsilon-d\Phi}_{L^2}$ is bounded and $\norm{\Phi_\epsilon - \Phi}_{L^2} \to 0$.
  Therefore by~\eqref{ineq:quadr-form-bound},
  \begin{equation}\label{ineq:from-quadr-form-bound}
    \begin{gathered}
      \int \abs{\Phi_\epsilon - \Phi}^2 \phi^2d\mu_\epsilon
      \leq \int \abs{\phi d(\Phi_\epsilon - \Phi) +
      (\Phi_\epsilon - \Phi)d\phi}^2 dv_g
      \\\leq \int \abs{d\Phi_\epsilon - d\Phi}^2 \phi^2dv_g
      +o(1)
    \end{gathered}
  \end{equation}
  as $\epsilon \to 0$. In a similar way,
  \begin{equation}
    \abs{\int \brt{d\Phi_\epsilon, (\Phi_\epsilon - \Phi)d\phi^2}dv_g}
    \leq \int \abs{d\phi^2} \abs{d\Phi_\epsilon}\abs{\Phi_\epsilon - \Phi}dv_g
    \to  0
  \end{equation}
  So, testing the equation $\Delta \Phi_\epsilon = \Phi_\epsilon \mu_\epsilon$
  against $(\Phi_\epsilon - \Phi)\phi^2$, one obtains
  \begin{equation}\label{ineq:from-eigen-func-eq}
    \int \brt{d\Phi_\epsilon, d(\Phi_\epsilon - \Phi)}\phi^2dv_g
    = \int\brt{\Phi_\epsilon, \Phi_\epsilon - \Phi}\phi^2d\mu_\epsilon
    +o(1).
  \end{equation}
  Now, one can
  multiply~\eqref{ineq:from-eigen-func-eq} by $2$ and sum
  it with~\eqref{ineq:from-quadr-form-bound} resulting
  in the inequality
  \begin{equation}
    \int \brr{\abs{d\Phi_\epsilon}^2 - \abs{d\Phi}^2}\phi^2dv_g
    \leq \int \brr{\abs{\Phi_\epsilon}^2 - \abs{\Phi}^2}\phi^2d\mu_\epsilon
    +o(1) \leq o(1),
  \end{equation}
  and the proof is complete.
\end{proof}
\begin{remark}\label{rem:h2-conv-simple}
  Lemma~\ref{lem:strong-conv} is designed to be applied to a maximizing sequence
  that approaches the limiting sphere-valued map \textit{from below},
  $\abs{\Phi_\epsilon}\nearrow 1=\abs{\Phi}$, which exists thanks
  to Proposition~\ref{prop:max-sequence} or~\cite*[Theorem~2.14]{Karpukhin-Nadirashvili-Penskoi-Polterovich:2022:existence}).
  It is also possible to prove $H^1$-convergence
  around \emph{stable} (a.k.a. \emph{good}) points for a maximizing sequence approaching
  the sphere-valued map \textit{from above}, $\abs{\Phi_\epsilon}\searrow 1=\abs{\Phi}$,
  as was done in~\cite[Claim~5]{Petrides:2018:exist-of-max-eigenval-on-surfaces,
  Petrides:2019:max-steklov-eigenval-on-surfaces}.
  However, it appears that the convergence \textit{from below} (the~\ref{it:lower-than}nd condition
  of Lemma~\ref{lem:strong-conv}) results in a shorter and simpler proof.
\end{remark}

\begin{lemma}\label{lem:regularity-in-good-pt}
  Let $\Sigma$ be a Riemannian surface,
  $\mu_\epsilon \overset{w^*}{\to} \mu \in \mathcal{M}^+(\Sigma)$
  and $\lambda_k(\mu_\epsilon) = 1$. Let
  $\Phi_\epsilon \overset{w}{\to} \Phi$ in
  $H^1(\Sigma,\R^n)$ with the properties~(1-2)
  of Lemma~\ref{lem:strong-conv}. Then
  \begin{equation}
    \Delta \Phi = \Phi \mu^c,
  \end{equation}
  where $\mu^c$ denotes the continuous part of $\mu$.
  Moreover,
  \begin{itemize}
    \item $\partial\Sigma = \varnothing$
    and $\abs{\Phi} \equiv 1$ on $\Sigma \implies$
    $\Phi \in C^\infty(\Sigma, \mathbb{S}^{n-1})$ is
    a smooth harmonic map and
    \begin{equation}
      \mu^c = \abs{d\Phi}_g^2\dv{g}.
    \end{equation}
    \item $\supp \mu \subset \partial\Sigma$
    and $\abs{\Phi} \equiv 1$ on $\partial\Sigma \implies$
    $\Phi \in C^\infty(\Sigma, \mathbb{B}^{n})$ is
    a smooth free boundary harmonic map, i.e.
    $\partial_\nu\Phi \perp T_{\Phi}\mathbb{S}^{n-1}$, and
    \begin{equation}
      \mu^c = \abs{\partial_\nu\Phi}\ds{g}.
    \end{equation}
  \end{itemize}
\end{lemma}
\begin{proof}
  Let us decompose $\mu$ into its continuous and atomic parts
  \begin{equation}\label{eq:meas-decomp}
    \mu=\mu^c+\sum_{p \in A(\mu)}w_p\delta_p,
  \end{equation}
  where $w_p > 0$ and $A(\mu)$ is the set of atoms of $\mu$.
  If $D$ is the collection of unstable points of $\brc{\mu_\epsilon}$,
  inequality \eqref{ineq:quadr-form-bound} implies that
  $A(\mu) \subset D$.

  Let us take $\phi, \psi \in C^\infty_{co}(\Sigma\setminus D)$
  such that $\psi \equiv 1$ on $\supp \phi$.
  By using the previous lemma,
  we see that
  \begin{multline}
    \int \langle d\Phi,d\phi\rangle dv_g=\lim_{\epsilon\to 0}
    \int \phi \psi \Phi_{\epsilon} d\mu_{\epsilon}=\lim_{\epsilon\to 0}
    \int \phi \brr{\psi\Phi_{\epsilon} - \psi \Phi}d\mu_\epsilon
    \\+\lim_{\epsilon\to 0}
    \int \phi \brr{\psi \Phi}d\mu_\epsilon = \int \phi \Phi d\mu,
  \end{multline}
  where  $\psi\Phi_{\epsilon} - \psi \Phi \to 0$ strongly
  in $H^1_{co}$ and $\mu_{\epsilon} \overset{w^*}{\to} \mu$
  as bilinear forms on $H^1_{co}$.
  Thus,
  \begin{equation}\label{eigen.map}
    \Delta \Phi = \Phi \mu^c \text{ on } \Sigma
  \end{equation}
  since discrete sets have zero capacity.

  If $\partial \Sigma =\varnothing$ and
  $\abs{\Phi} \equiv 1$ on $\Sigma$, one has $\Phi \cdot d\Phi = 0$, and
  after testing equation~\eqref{eigen.map} against $\varphi \Phi \in H^1_{co}$,
  we conclude that
  $$
  \mu^c=|d\Phi|^2 \dv{g}
  $$
  and
  \begin{equation}\label{eq:harmonic}
    \Delta \Phi = |d\Phi|^2 \Phi
  \end{equation}
  in the weak sense. Thus, $\Phi\colon \Sigma \to \mathbb{S}^{n-1}$
  is a weakly harmonic map, which is in fact smooth by
  Hélein's regularity theory~\cite{Helein:1991:weakly-harm}.

  If $\abs{\Phi} \equiv 1$ on $\partial\Sigma$ and
  $\supp \mu \subset \partial\Sigma$, we notice that
  $\Phi$ is harmonic in the interior, in particular
  $\abs{\Phi} \leq 1$. Then we
  take a map $Y\in H^1_{co}(\Sigma,\mathbb{R}^n)$ such that
  $Y(x) \in T_{\Phi(x)}\mathbb{S}^{n-1}$ for a.e.
  $x \in \partial\Sigma$,
  i.e. $\langle Y,\Phi\rangle|_{\partial\Sigma} = 0$.
  Equation~\eqref{eigen.map} implies
  $$
  \int \langle d\Phi,dY\rangle dv_g=
  \int_{\partial\Sigma} \langle Y,\Phi\rangle d\mu^c = 0,
  $$
  i.e. $\Phi:\Sigma\to \mathbb{B}^n$ is a weakly free boundary harmonic map, hence smooth
  by~\cite[Theorem~1.2]{Laurain-Petrides:2017:free-bndry-harmonic} or
  by by~\cite[Theorem~2.2]{Scheven:2006:regularity-free-bndry-harm-maps}
  (note that in dimension 2, the stationarity
  condition is unnecessary since the monotonicity property
  is trivially satisfied).
  At the same time, for harmonic functions, one has
  \begin{equation}
    \Delta \Phi = \partial_\nu \Phi \ds{g}.
  \end{equation}
  Comparing this equation with~\eqref{eigen.map} and multiplying
  both of them on $\Phi$, we obtain
  $$
  \mu^c= \langle \partial_\nu \Phi,\Phi\rangle \ds{g}
  \quad\text{and}\quad \partial_\nu\Phi \parallel \Phi.
  $$
\end{proof}

\subsection{Proof of Theorem~\ref{thm:main} and Theorem~\ref{thm:main-stek}}

Let us notice that for the theorems to be proved, it suffices to establish
the following equalities,
\begin{gather}
  \Lambda_k(\Sigma, [g];\G) = \max_{
    \substack{b\in\mathbb{N}_\G(\Sigma) \\ 0 \leq b \leq k}}
  \brc{\overline{\lambda}_{k-b}(\abs{d\Phi}_g^2 g) + 8\pi b},
  \\
  \mathfrak{S}_k(\Sigma, [g];G) =  \max_{
  \substack{b\in\mathbb{N}_\G(\Sigma) \\ 0 \leq b \leq k}}
  \brc{\overline{\sigma}_{k-b}(\abs{\partial_\nu\Phi} g) + 2\pi b}.
\end{gather}
for some smooth $\G$-equivariant (respectively, free boundary) harmonic map
to $\Sph^n$ (respectively, to $\mathbb{B}^n$). We treat the lower and upper bounds separately.

\subsubsection{Lower bound}

The lower bounds on $\Lambda_k(\Sigma,[g];\G)$ and $\mathfrak{S}_k(\Sigma,[g];\G)$,
i.e. the inequalities $\geq$ in the corresponding theorems,
were essentially proved by the gluing techniques described
in~\cite{Colbois-ElSoufi:2003:extremal}
and~\cite{Fraser-Schoen:2020:steklov-unions} so that the spectrum
convergences to the spectrum of the disjoint union, and we can attach
up to $k$ spheres (or disks). In order to obtain $\G$-invariant
metrics, it is enough to make sure that the set of
the gluing neighborhoods is itself $\G$-invariant.

Another approach is to proceed as follows. Fix a number $q \le k$ such that
$$
  q = \sum_i m_i \abs{\G\!\cdot\!p_i},
$$
where $m_i \ge 1$. Around each point $p = p_i$, consider
$m = m_i$ spheres $M_j\in\brc{\mathbb{S}^2, \mathbb{D}^2}$
(or disks if $p\in\partial\Sigma$). We use coordinates coming from
a stereographic projection of $\mathbb{S}^2$, and identify
the disk $\mathbb{D}^2 \approx \mathbb{S}^2_+$
with the upper-half sphere. We are going to
construct a sequence
of $L^\infty$ measures $\mu_\epsilon = \rho_\epsilon\dv{g}$
(or $\mu_\epsilon = \rho_\epsilon\ds{g}$) that concentrates
around $p$ telescopically creating $m$ spheres that are
sequentially attached to each other. Locally on the corresponding
balls $B(r;M_j)$, we have a chain of
'homotheties'
\begin{equation}
  B(r_\epsilon^0;\Sigma) \overset{\pi_\epsilon^1}{\to}
  B(\alpha_\epsilon^1 r_\epsilon^{0};M_1) \overset{\pi_\epsilon^2}{\to}
  \cdots \overset{\pi_\epsilon^m}{\to}
  B(\alpha_\epsilon^{m}r_\epsilon^{m-1};M_m)
\end{equation}
such that $\pi_\epsilon^j(x) = \alpha_\epsilon^j x$ and
$r^j_\epsilon < \alpha_\epsilon^{j}r_\epsilon^{j-1} $, where
$x\in B(r_\epsilon^{j-1};M_{j-1})$,
$r_\epsilon^{j} \searrow  0$ and
$\alpha_\epsilon^j r_\epsilon^{j-1}\nearrow \infty$. For a given set of measures
$\mu^j \in \mathcal{M}^+(M_j) \cap L^\infty$ such that $\mu^0$ is
$\G$-invariant and $\mu^i$, $i\geq 1$, are invariant with respect to the stabilizer
$\G_p$, we
choose the measures
$\mu_\epsilon$ so that
\begin{align}
  \mu_\epsilon = \mu^0 &\text{ on }
    \Sigma \setminus B(r_\epsilon^0;\Sigma),
  \\\mu^1_\epsilon := (\pi_\epsilon^1)_*[\mu_\epsilon] =
    \mu^1  &\text{ on } B(\alpha_\epsilon^1
    r_\epsilon^0;M_1) \setminus B(r_\epsilon^1;M_1)
  \\ &\cdots
  \\ \mu^m_\epsilon := (\pi_\epsilon^{m})_*[\mu^{m-1}_\epsilon]
    = \mu^m &\text{ on } B(\alpha_\epsilon^m
  r_\epsilon^{m-1};M_1) \setminus B(r_\epsilon^m;M_m).
\end{align}
We
average this sequence by the action of $\G$ to get the same behavior around
the other points in the corresponding orbits. Then
the variational characterization of $\lambda_k$ and
the conformal invariance of Dirichlet energy imply that
\begin{equation}
  \lambda_k(\mu_\epsilon) \geq \lambda_k(\tilde{\Sigma}_\epsilon,\mu^0 \sqcup \mu^1
  \sqcup\cdots\sqcup \mu^m),
\end{equation}
where
\begin{equation}
  \tilde{\Sigma}_\epsilon:=\Sigma \setminus B(r_\epsilon^0;\Sigma)\sqcup\bigsqcup_{j>0}
B(\alpha_\epsilon^j r_\epsilon^{j-1};M_j) \setminus B(r_\epsilon^j;M_j)
\to \Sigma \sqcup\bigsqcup_{j>0} M_j.
\end{equation}
One also has
\begin{equation}
  \lambda_k(\tilde{\Sigma}_\epsilon, \mu^0 \sqcup \mu^1
  \sqcup\cdots\sqcup \mu^m)
  \to \lambda_k(\mu^0 \sqcup \mu^1\sqcup\cdots\sqcup \mu^m)
\end{equation}
by~\cite[Proposition~4.8]{Girouard-Karpukhin-Lagace:2021:continuity-of-eigenval},
where we use harmonic extension operators into the removed domains.
Now, if $\mu^i$, $i\geq 1$, are the volume measures of $\Sph^2$ (resp. $\partial\Sph^2_+$),
we will get the desired inequality from below.

\subsubsection{Upper bound}
Recalling the equalities~\eqref{eq:meas-def}, we
start from a maximizing sequence of Proposition~\ref{prop:max-sequence}
and renormalize $\mu_\epsilon$, so that
$\lambda_\epsilon = \lambda_k(\mu_\epsilon) = 1$. Then we apply
Lemma~\ref{lem:strong-conv} and Lemma~\ref{lem:regularity-in-good-pt}
on a collection of domains $\cup_i \Omega_i$ exhausting $\Sigma\setminus D$
(the \ref{it:lower-than}d condition of Lemma~\ref{lem:strong-conv} is trivially satisfied
since $\abs{\Phi_\epsilon} \leq 1 = \abs{\Phi}$).
Thus, the weak limit $\Phi\in H^1(\Sigma, \mathbb{R}^n)$
of the maps $\Phi_{\epsilon}$ defines a (free boundary)
harmonic map
$$
\Phi: \Sigma\to \mathbb{S}^{n-1},
\quad\text{or}\quad
\Phi: \Sigma\to \mathbb{B}^{n}
$$
Moreover, the limiting measure $\mu$ has the form
$$
\mu=|d\Phi|^2dv_g+\sum_{p \in A(\mu)}w_p\delta_p,
\quad\text{or}\quad
\mu=|\partial_\nu \Phi|ds_g+\sum_{p \in A(\mu)}w_p\delta_p.
$$

Since the set of the atoms $A(\mu)$ must be
$\G$-invariant, it splits into the union of some orbits of $\G$. That
is,
$$
  A(\mu) = \bigsqcup_{p \in A(\mu)} \G\!\cdot\!p.
$$
If $A(\mu) \neq \varnothing$, we handle the bubbles
as it is described in Definition~\ref{def:bubble-tree}.
It has been shown, see~\eqref{eq:L1-densities}, that for a
$L^1$-measure $\tilde{\mu}$ either on $\Sph^2$ or
on $\partial \mathbb{D}^2$, one has
\begin{equation}
  \overline{\lambda}_{k}\brr{\mathbb{S}^2,\tilde{\mu}}
  \leq \Lambda_k(\mathbb{S}^2) = 8\pi k,
\end{equation}
see~\cite{Karpukhin-Nadirashvili-Penskoi-Polterovich:2021:eigenval-on-sphere},
and
\begin{equation}
  \overline{\lambda}_{k}\brr{\mathbb{D}^2,\tilde{\mu}}
  \leq \mathfrak{S}_k(\mathbb{D}^2) = 2\pi k,
\end{equation}
see~\cite{Girouard-Polterovich:2010:payne-schiffer}. By Corollary~\ref{cor:union-ineq} and Corollary~\ref{cor:st-union-ineq}
we have the upper bounds
\begin{equation}
  \Lambda_k(\Sigma,[g];\G) \le \overline{\lambda}_{k_0}\brr{[g],|d\Phi|^2dv_g} +
  \sum_{i=1}^b \overline{\lambda}_{k_i}\brr{\mathbb{S}^2,\mu_i}
\end{equation}
and
\begin{equation}
  \mathfrak{S}_k(\Sigma,[g];\G) \le \overline{\lambda}_{k_0}\brr{[g],|\partial_\nu \Phi|ds_g} +
  \sum_{i=1}^b \overline{\lambda}_{k_i}\brr{\mathbb{D}^2,\mu_i}
\end{equation}
provided $\mu_i \in L^1$, where $\sum_{i=0}^b k_i = k$ for some $k_0 \in \mathbb{N}$.
Note that $b \in \mathbb{N}_{\G}$ because the behavior of $\mu_\varepsilon$
around $\tilde{p} \in A(\mu)$ is
the same for all $\tilde{p} \in \G\!\cdot\!p$.
Thus, the following claim finishes the proof of the upper bound.
\begin{claim}
  If one starts from a maximizing sequence given
  by Proposition~\ref{prop:max-sequence}, one obtains
  $L^1$ bubble measures $\mu_i$.
\end{claim}
\begin{proof}
  Recall from Appendix~\ref{app:blow-up} that $\mu_i$
  are the continuous parts of the
  weak$^*$ limits
  of $\tilde{\mu}_\epsilon := (T_\epsilon)_*[\mu_\varepsilon]$,
  \begin{equation}
    T_\epsilon(x) = \frac{x - p_\epsilon}{\alpha_\epsilon}.
  \end{equation}
  The idea of the proof is the same as
  in~\cite[Proposition~5.5]{Karpukhin-Nadirashvili-Penskoi-Polterovich:2022:existence},
  and it is to show that the measures $\tilde{\mu}_\epsilon$
  have properties~(1–3) of Proposition~\ref{prop:max-sequence} on
  any bounded ball of $\R^2$ (or $\R^2_+$). The push-forward equation
  \begin{equation}
    \Delta_{\mathbb{S}^2} \tilde{\Phi}_\epsilon =
    \tilde{\Phi}_\epsilon \tilde{\mu}_\epsilon,
    \text{  where  }
    \tilde{\Phi}_\epsilon(x) = \Phi_\epsilon(\alpha_\epsilon x + p_\epsilon),
    \ x \in B(R;0)
  \end{equation}
  obviously holds if we are using isothermal coordinates
  on $\Sigma$.
  Then the norms $||d\tilde{\Phi}_\epsilon||_{L^2[B(R;0)]}$ are conformally
  invariant, hence uniformly bounded. So one can extract a
  weakly converging subsequences
  $\tilde{\Phi}_\epsilon \oset{w}{\to} \tilde{\Phi}$ and
  $\tilde{\mu}_\epsilon \to \tilde{\mu}$ and
  apply Lemma~\ref{lem:strong-conv} and Lemma~\ref{lem:regularity-in-good-pt}
  with $\Sigma = B(R;0)$ to show the regularity of $\tilde{\mu}$.

  Since $T_\epsilon[B(r_\epsilon)]$
  exhaust $\R^2$,
  only the \ref{it:almost-one}d property of Proposition~\ref{prop:max-sequence}
  is not obvious. Notice that
  \begin{equation}
    T_\epsilon^* g_{\mathbb{S}^2} =
    \brr{\frac{\alpha_\epsilon}
      {\alpha_\epsilon^2 + (x - p_\epsilon)^2}}^2 g_{\R^2}.
  \end{equation}
  \textbf{The case }$\mathbf{\partial \Sigma = \varnothing}$.
  Then $\mu_\epsilon = \rho_\epsilon \dv{g}$ and
  $\tilde{\mu}_\epsilon = \tilde{\rho}_\epsilon \dv{\Sph^2}$, and one has
  \begin{equation}
    T_\epsilon^*\tilde{\rho}_\epsilon =
    \brr{1 + \brr{\frac{x - p_\epsilon}{\alpha_\epsilon}}^2}^2
    \alpha_\epsilon^2 \rho_\epsilon\brr{\frac{\dv{g}}{\dv{\R^2}}},
  \end{equation}
  where $\frac{\dv{g}}{\dv{\R^2}} \to 1$ on $B(r_\epsilon)$.
  On a bounded ball $\frac{x - p_\epsilon}{\alpha_\epsilon} =: y  \in B(R;0) \subset \R^2$, we have
  \begin{equation}
    \tilde{\rho}_\epsilon(y) \leq c_2 \alpha_\epsilon^2 \rho_\epsilon(x),
  \end{equation}
  where $c_2 = c_2(R)$.
  If one denotes $\tilde{\epsilon} := \epsilon \alpha^{-2}_\epsilon$,
  then there are two options.
  \begin{enumerate}
    \item The sequence $\tilde{\epsilon}\in \brc{\tilde{\epsilon}_i}$ is bounded
    from below,
    i.e. $\tilde{\epsilon} \geq c > 0$.
    Then
    \begin{equation}
      \norm{\tilde{\mu}_\epsilon}_{L^\infty}\leq c_2 \tilde{\epsilon}^{-1}
    \norm{\epsilon \mu_\epsilon}_{L^\infty} \leq c_3(R)
    \end{equation}
    since
    $\norm{\epsilon \mu_\epsilon}_{L^\infty}$ is bounded by
    Proposition~\ref{prop:max-sequence}.
    Thus, $\forall R > 0 \,$ $\tilde{\mu}_\epsilon \overset{w^*}{\to} \tilde{\mu}$
    in $L^\infty[B(R,0)]$ $\implies
    \tilde{\mu} \in L^1(\Sph^2)$.
    \item Up to a subsequence, $\tilde{\epsilon}\to0$.
    Thus, the sequences $\tilde{\mu}_{\tilde{\epsilon}}:=\tilde{\mu}_{\epsilon}$
    and $\tilde{\Phi}_{\tilde{\epsilon}}:=\tilde{\Phi}_\epsilon$
    have the properties on $B(R,0)$ similar to those of
    Proposition~\ref{prop:max-sequence} because
    \begin{equation}
      \begin{split}
        &v_{\Sph^2}(\{x \in B(R,0) \colon
        |\tilde{\Phi}_{\tilde{\epsilon}}(x)| < 1\})
      \\ &= c_2 \alpha_\epsilon^{-2} v_{g}\brr{\brc{x \in B(\alpha_\epsilon R, p_\epsilon) \colon
      |\Phi_{\epsilon}(x)| < 1}}
      \\ &\leq c_2 \alpha_\epsilon^{-2} \epsilon
      \leq c_2 \tilde{\epsilon} \to 0.
      \end{split}
    \end{equation}
    Each $B(R,0)$ has a at most $k$ unstable points, or else
    we can argue as in Remark~\ref{rem:size_of-bad}.
    So, by~Lemma~\ref{lem:strong-conv} and Lemma~\ref{lem:regularity-in-good-pt},
    the continuous part of $\tilde{\mu}$ is smooth in that case.
  \end{enumerate}
  \textbf{The case }$\mathbf{\partial \Sigma \neq \varnothing}$
  is handled similarly. Now, $\mu_\epsilon = \rho_\epsilon \ds{g}$ and
  $\tilde{\mu}_\epsilon = \tilde{\rho}_\epsilon \ds{\Sph^2_+}$. The estimate
  is slightly different,
  \begin{equation}
    \tilde{\rho}_\epsilon(y) \leq c_1 \alpha_\epsilon \rho_\epsilon(x).
  \end{equation}
  And again, independently on wether $\tilde{\epsilon} := \epsilon \alpha^{-1}_\epsilon
  \to 0$ or not, one proves that $\tilde{\mu}\in L^1(\partial \mathbb{D}^2)$.
\end{proof}
% \newpage

\section{Proofs of applications}\label{sec:proof-of-app}

\subsection{On computations of \texorpdfstring{$\mathfrak{S}_k(\mathbb{D}^2;\G)$}{Lg}}

We will need a couple of results concerning
free boundary harmonic maps $\Phi\colon \mathbb{D}^2 \to \mathbb{B}^n$.
Such maps are weakly conformal hence free boundary branched minimal
immersions. The argument goes as follows. One wants to prove that
the holomorphic function
$f(z):= z^2\brt{\partial_z \Phi, \partial_z\Phi}$ is zero. In fact,
it is enough to show that it is constant since $f(0)=0$.
The free boundary condition $\partial_r \Phi \bot  \partial_{\theta}\Phi$ implies
that $f$ is real on $\partial \mathbb{D}^2$, Then the symmetry
principle and Liouville's theorem lead to the result.
\begin{proposition}[Theorem~2.1 in~\cite{Fraser-Schoen:2015:unique-free-disks}]
  \label{prop:disk-and-its-cover}
  Let $\Phi\colon \mathbb{D}^2 \to \mathbb{B}^n$ be a smooth
  free boundary harmonic map. Then $\Phi(\mathbb{D}^2)$ is
  an equatorial plane disk, so $\Phi\colon \mathbb{D}^2 \to
  \mathbb{D}^2 \subset \mathbb{B}^n$ is a branched covering of
  the unit disk.
\end{proposition}
Consider a smooth free boundary harmonic map $\Phi\colon
\mathbb{D}^2 \to \mathbb{D}^2$. Since it is
conformal, we can view it as a holomorphic function (after the appropriate choice of orientations on $\mathbb{D}^2$),
and one has $\abs{\partial_r \Phi} = \abs{\partial_\theta \Phi}$.
Hence, we see that the induced boundary density equals
\begin{equation}\label{eq:conf-condition}
  \rho = \brt{\partial_r \Phi, \Phi} =
  \abs{\partial_r \Phi} = \abs{\partial_\theta \Phi}.
\end{equation}
By $\sigma_k(\rho)$ we denote the $k$-th steklov eigenvalue
with the density $\rho d\theta$, i.e. $\sigma_k(\rho) = \lambda_k(\rho d\theta)$
in the measure-theoretical sense.
\begin{lemma}\label{lem:best-cover-disk}
  Let $\Phi\colon \mathbb{D}^2 \to \mathbb{D}^2$ be a holomorphic
  branched covering. Clearly, it solves the Steklov problem with the
  density $\rho = \abs{\partial_\theta \Phi}$ and $\sigma = 1$. For any
  $k$ such that $\sigma_k\brr{\abs{\partial_\theta \Phi}} = 1$, one has
  \begin{equation}
    \overline{\sigma}_k\brr{\abs{\partial_\theta \Phi}} \leq
    \overline{\sigma}_k(1),
  \end{equation}
  and the equality is achieved if and only if $\Phi$ is a conformal automorphism
  of the unit disk.
\end{lemma}
\begin{proof}
For a holomorphic branched covering $\Phi\colon
\mathbb{D}^2 \to \mathbb{D}^2$ of the closed unit disk, we define
  \begin{equation}
    \ind \Phi =  \max \brc{\dim V\colon V \subset C^\infty(\mathbb{T})
    \text{ s.t. }Q_\Phi|_V < 0},
  \end{equation}
  where $\mathbb{T} = \partial \mathbb{D}^2$,
  \begin{equation}
    Q_\Phi(u) = \int_{\mathbb{T}} \brr{u\partial_r \hat{u} - \abs{\partial_r \Phi} u^2} d\theta,
  \end{equation}
  and $\hat{u}$ is the harmonic extension of $u \in
  C^\infty(\mathbb{T})$ inside $\mathbb{D}^2$.
  From the definition of the index, it's clear that the equality
  $\sigma_k(\abs{\partial_r \Phi}) = 1$ implies
  \begin{equation}\label{ineq:index-number}
    \ind \Phi \leq k
  \end{equation}
  since we count eigenvalues starting from $k=0$.

  A holomorphic map of the unit disk to itself
  that extends continuously to the boundary
  $\Phi\colon \mathbb{T} \to \mathbb{T}$ is called a finite
  Blaschke product, and is in fact expressed as
  \begin{equation}
    \Phi(z) = e^{i\theta_0} \prod_{i=1}^s
    \brr{\frac{z - a_i}{1 - \overline{a}_i z}}^{m_i},
  \end{equation}
  where $m_i \geq 0$, $d = \deg \Phi|_{\mathbb{T}} = \sum m_i$ and $\abs{a_i} < 1$.
  Notice that any subproduct $\Phi'$ of $\Phi$ is itself a (finite)
  Blaschke product. In that case, we write
  \begin{equation}
    \Phi' \subset \Phi.
  \end{equation}

  First, by similar reasoning as for equation~\eqref{eq:conf-condition},
  any $\Phi' \subset \Phi$ is conformal,
  $\partial_\theta \Phi'|_\mathbb{T} \in T \mathbb{T}$, so
  $\partial_r \Phi' \parallel \Phi'$ and
  $\abs{\partial_\theta \Phi'} = \abs{\partial_r \Phi'}$.
  Therefore, $\Phi'$ solves the appropriate Steklov problem
  \begin{equation}
    \partial_r \Phi' = \abs{\partial_\theta \Phi'} \Phi'.
  \end{equation}
  In particular,
  \begin{equation}\label{eq:sigma_k}
    \overline{\sigma}_k(\abs{\partial_r \Phi}) = 1 \cdot
    \int_{\mathbb{T}} \abs{\partial_\theta \Phi'} d\theta
     = 2\pi d'.
  \end{equation}

  Second, the holomorphic derivative $\partial_z \Phi'$
  is never zero on $\mathbb{T}$. Otherwise, $\Phi'$ would have
  a branched point and would
  map some points of the interior of the disk outside of it.
  Since $\Phi'$ is holomorphic, $\arg \Phi'$ rotates
  counterclockwise, so
  \begin{equation}
    \partial_\theta \arg \Phi'|_{\mathbb{T}} > 0.
  \end{equation}

  Let us decompose $\Phi$ into two of
  its subproducts $\Phi = \Phi_1 \Phi_2$. From the discussion
  above, it follows that
  \begin{equation}
    \abs{\partial_\theta \Phi} =
    \partial_\theta \arg \Phi =
    \partial_\theta \arg \Phi_1 + \partial_\theta \arg \Phi_2
    = \abs{\partial_\theta \Phi_1} + \abs{\partial_\theta \Phi_2}
  \end{equation}
  on $\mathbb{T}$. Let us set either $u = \operatorname{Re} \Phi'$
  or $u=\operatorname{Im}\Phi'$ for $\Phi' \subset \Phi$ and calculate
  \begin{equation}
    Q_\Phi(u) = \int_{\mathbb{T}} \brr{\abs{\partial_\theta \Phi'}
    - \abs{\partial_\theta \Phi}}u^2 d\theta < 0.
  \end{equation}
  Hence
  \begin{equation}
    \ind \Phi \geq 2 \abs{\brc{\Phi' \colon \Phi' \subsetneq \Phi}} - 1
  \end{equation}
  Different $\Phi'$ are linearly independent since each of them has
  a unique configuration of poles. By elementary combinatorics,
  \begin{equation}\label{ineq:cover-index}
    \ind \Phi \geq - 3 + 2\prod_{i=1}^s \brr{m_i + 1} =
    -1 + 2d + 2\sum_{i < j } m_i m_j + \cdots
    \geq 2d - 1.
  \end{equation}
  Since $\ind z^d$ is just the number of eigenvalues of the
  unit disk~\eqref{eq:disk-spec} less than $d$,
  we see that $\ind z^d = 2d - 1$ and
  $\overline{\sigma}_{2d-1}(1)  = 2\pi d$. Combining
  these facts with~\eqref{eq:sigma_k} and~\eqref{ineq:index-number},
  we obtain
  \begin{equation}\label{ineq:proof-cover}
    \overline{\sigma}_k(\abs{\partial_r \Phi}) =
    2\pi d = \overline{\sigma}_{2d-1}(1) \leq
    \overline{\sigma}_{k}(1).
  \end{equation}

  If $\Phi$ is not an automorphism of the disk,
  then $s \geq 2$ in~\eqref{ineq:cover-index}, so one has
  $\sum_{i < j} m_i m_j \geq 1$ and consequently
  \begin{equation}
    \ind \Phi \geq 2d + 1,
  \end{equation}
  which implies the strict inequality in~\eqref{ineq:proof-cover}.
\end{proof}

\subsubsection{The proof of Theorem~\ref{thm:symmetric-disk}}
  First of all, if the maximum for $\overline{\sigma}_k$ is
  achieved on a $L^\infty$ density $\rho$, i.e. if
  $\overline{\sigma}_k(\rho) = \mathfrak{S}_k(\mathbb{D}^2;\G)$,
  the proofs of Proposition~\ref{prop:max-sequence} and
  Lemma~\ref{lem:regularity-in-good-pt} imply
  that $\rho$ comes from a free boundary harmonic map. Then
  Proposition~\ref{prop:disk-and-its-cover} and Lemma~\ref{lem:best-cover-disk}
  show that up to a conformal automorphism $\rho$ is a constant,
  say $\rho \equiv 1$.
  Since $\overline{\sigma}_k(1) = 2\pi\floor*{\frac{k+1}{2}}$,
  as long as the maximum is achieved, it equals
  \begin{equation}\label{eq:smooth-max-disk}
    2\pi\floor*{\frac{k+1}{2}}.
  \end{equation}
  Then we inductively prove that
  \begin{equation}\label{eq:disk-induction-state}
    \mathfrak{S}_{mn + r}(\mathbb{D}^2;\G) = 2\pi\brr{mn + \floor*{\frac{r+1}{2}}},
  \end{equation}
  where $r < n$.

  Equation~\eqref{eq:disk-induction-state} is true
  when $m = 0$, as one easily sees from Corollary~\ref{cor:stek-non-babl}.
  Suppose we proved it for all $m < m'$. Since in our case,
  $\mathbb{N}_{\G} = n\mathbb{N}$, Theorem~\ref{thm:main-stek} implies
  that either
  \begin{multline}
    \mathfrak{S}_{m'n + r}(\mathbb{D}^2;\G)
    \\=
    \max_{0 \leq m < m'} \brr{\mathfrak{S}_{mn + r}(\mathbb{D}^2;\G)
    + 2\pi n (m'- m)}
    = 2\pi \brr{m'n + \floor*{\frac{r+1}{2}}},
  \end{multline}
  where the last equality follows from the induction assumption,
  or the supremum is achieved on a smooth density and
  $\mathfrak{S}_{m'n + r}(\mathbb{D}^2;\G)$ is strictly greater
  than $2\pi \brr{m'n + \floor*{\frac{r+1}{2}}}$. In the latter case,
  we have a contradiction with equation~\eqref{eq:smooth-max-disk}
  because it is impossible that
  \begin{equation}
    \floor*{\frac{m'n + r+1}{2}} >  m'n + \floor*{\frac{r+1}{2}}.
  \end{equation}

\newpage
\subsection{On computations of \texorpdfstring{$\Lambda_k(\mathbb{S}^2;\G)$}{Lg}}
Let us recall some basic facts about
harmonic maps between spheres. First, any harmonic map
$\Phi\colon \mathbb{S}^2 \to \mathbb{S}^n$ is weakly conformal (i.e.
$\Phi^* g_{\Sph^n} = \lambda(z) g_{\Sph^2}$, $\lambda \geq 0$) hence branched minimal
immersion. To see this, we notice that $\Sph^2$
has no nonzero holomorphic quadratic differentials, but for harmonic maps,
$\partial_z \bar{\partial}_z\Phi \parallel \Phi$, and therefore,
$\brt{\partial_z \Phi, \partial_z\Phi}dz^2 = 0$.
\begin{proposition}[\cite{Calabi:1967:min-imersions}, \cite{Barbosa:1975:min-imersions}]
  \label{prop:harm-sheres}
  Let $\Phi\colon \mathbb{S}^2 \to \mathbb{S}^n$ be a smooth harmonic
  map. Suppose, that $\Phi$ is linearly full, i.e. it does not lie in
  any smaller subsphere $\mathbb{S}^{n-1} \subset \mathbb{S}^n$. Then
  \begin{itemize}
    \item $n$ is even, say $n = 2m$;
    \item $\Area (\Phi^* g_{\Sph^n}) = 4\pi d$, and $d \in \mathbb{N}$
    is called the harmonic degree of $\Phi$;
    \item $d \geq \frac{1}{2}m(m+1)$.
  \end{itemize}
\end{proposition}
Let $Harm_d^f(\Sph^2, \Sph^{2m})$ be the space of all (smooth)
linearly full harmonic
maps between $\Sph^2$ and $\Sph^{2m}$ with area $4\pi d$. Then

\begin{proposition}[\cite{Barbosa:1975:min-imersions}]
  \label{prop:min-degree-maps}
  The pairs
  $\brc{\pm \Phi} \subset Harm_d^f(\Sph^2, \Sph^{2m})$ are in one-to-one
  correspondence with linearly full totally isotropic holomorphic
  curves $\xi\colon\Sph^2 \to \mathbb{CP}^{2m}$. The curve
  $\xi$ is referred as the directrix curve of $\Phi$. Moreover
  \begin{itemize}
    \item $\Phi$ is proportional to $\xi_{m-1}\wedge \overline{\xi}_{m-1}$,
     where
    $\xi_{m-1} = \xi\wedge\partial_z\xi\wedge\cdots\wedge\partial_z^{m-1}\xi.$
    \item $\Phi, \tilde{\Phi}\in Harm_d^f(\Sph^2, \Sph^{2m})$ are isometric
    if and only if their directrix curves $\xi, \tilde{\xi}$ are isometric.
    \item If $d_0 = \frac{1}{2}m(m+1)$, the group $SO(\mathbb{C}^{2m+1})$
    acts freely and transitively on $Harm_{d_0}(\Sph^2, \Sph^{2m})$
    by $\xi \mapsto A \xi$. Thus,
    \begin{equation}
      SO(\mathbb{C}^{2m+1}) \approx Harm_{d_0}(\Sph^2, \Sph^{2m})/\brc{\pm 1}.
    \end{equation}
  \end{itemize}
\end{proposition}
The next proposition is a consequence of~\cite[Theorem~9]{Calabi:1953:isometric-imbedding},
see also (4.12) in~\cite{Griffiths:1974:lie-groups-and-moving-frames}.
\begin{proposition}
  \label{prop:curve-ridgidity}
  Let $\xi,\zeta\colon\mathbb{S}^2\to\mathbb{CP}^n$ be two
  isometric linearly full holomorphic curves, i.e. $\xi^*g_{FS}=\zeta^*g_{FS}$,
  where $g_{FS}$ is a Fubini–Study metric on $\mathbb{CP}^n$. Then
  there exists a unique $U \in PSU(\mathbb{CP}^n)$ such that
  $\zeta = U \xi$.
\end{proposition}
\begin{lemma}\label{lem:equiv-degree}
  If $\Phi \in Harm_d^f(\Sph^2, \Sph^{2m})$ is $\G$-equivariant,
  where $\G \subset O(\R^3)$,
  one has
  \begin{equation}
    2d \in m(m+1) + \mathbb{N}_\G.
  \end{equation}
\begin{proof}
  A harmonic
  map $\Phi \in Harm_d^f(\Sph^2, \Sph^{2m})$ corresponds to
  linearly full totally isotropic (anti-)holomorphic
  curves
  \begin{equation}
    \xi, \overline{\xi}\colon\Sph^2 \to \mathbb{CP}^{2m}.
  \end{equation}
  At the same time, equations~(6.4)–(6.6) of~\cite{Barbosa:1975:min-imersions} imply
  \begin{equation}
    2d = \frac{\Area \Phi}{2\pi} = m(m-1) + \sum_{s=1}^m\sum_{j=1}^s
    \sigma_{m-j},
  \end{equation}
  where $\sigma_{k} = \sum_{z\in\mathbb{S}^2}\sigma_{k}(z)$ is
  the degree of the ramification divisor of $\xi_k$, i.e.
  \begin{equation}
    \xi_{k}^* (g_{FS}) = \abs{z - z_0}^{2\sigma_{k}(z_0)} h(z) dz\otimes d\bar{z}
  \end{equation}
  and $h(z_0)\neq 0$.

  Let us take $\gamma \in \G_+ = \G \cap SO(\R^3)$. Then $\gamma^*\Phi = \Phi \circ \gamma$
  corresponds to the holomorphic curve $\gamma^* \xi$ because $\gamma^*\Phi$ is proportional to $\gamma^*\xi\wedge
  \overline{\gamma^*\xi}$. An easy computation shows that
  \begin{equation}
    (\gamma^* \xi)_k =  (\partial_z\gamma)^k\gamma^* \xi_k.
  \end{equation}
  Since $\Phi$ and $\gamma^*\Phi$ are isometric, it follows from
  Propositions~\ref{prop:min-degree-maps} and~\ref{prop:curve-ridgidity}
  that $\gamma^* \xi = U_\gamma \xi$, so
  \begin{equation}
    (\gamma^* \xi)_k =  (\wedge^k U_\gamma) \xi_k.
  \end{equation}
  Combining these two equations, we see  that $\sigma_k(\gamma(z)) = \sigma_k(z)$.

  If $\gamma \in \G \setminus \G_+$, then $\gamma^*\xi$ becomes
  anti-holomorphic, and $\overline{\gamma^*\xi}$ becomes holomorphic.
  In that case, the directrix curve of
  $\gamma^*\Phi$ is $\overline{\gamma^*\xi}$, and the arguments as above show
  \begin{equation}
    \overline{(\bar{\partial}_z\gamma)^k\gamma^* \xi_k} = \brr{\overline{\gamma^*\xi}}_k
    = (\wedge^k U_\gamma) \xi_k.
  \end{equation}
  Again, we have $\sigma_k(\gamma(z)) = \sigma_k(z)$. After recalling
  the definition of $\mathbb{N}_\G$~\eqref{eq:N_G}, we thus obtain
  $\sigma_k \in \mathbb{N}_\G$, which finishes the proof.
\end{proof}

\end{lemma}
\begin{definition}
  By the (spectral) index of a harmonic map $\Phi\colon (\Sigma,[g]) \to \mathbb{S}^n$ we mean the index of
  the corresponding quadratic form
  \begin{equation}
    \ind \Phi =  \max \brc{\dim V\colon V \subset C^\infty(\Sigma)
    \text{ s.t. }Q_\Phi|_V < 0},
  \end{equation}
  where
  \begin{equation}
    Q_\Phi(u) = \int \brr{\abs{du}_g^2 - \abs{d\Phi}_g^2 u^2} dv_g.
  \end{equation}
  We also define
  \begin{equation}
    \nul \Phi = \dim \ker \brr{\Delta_g - \abs{d\Phi}^2_g}.
  \end{equation}
  Note that the definitions of $Q_\Phi$, $\ind \Phi$ and $\nul \Phi$
  depend only on $[g]$. Also, $\Sph^2$ has the unique conformal class only up to a diffeomorphism,
  so we consider only this conformal class (denoted by $[g_{\Sph^2}]$).
\end{definition}

If we combine Propositions 3.12 and 3.16 of~\cite{Karpukhin:2021:index-of-min-shperes},
we immediately get
\begin{proposition}
  \label{prop:ind-est}
  For any $\Phi \in Harm_d^f(\Sph^2, \Sph^{2m})$ one has
  \begin{equation}
    \ind \Phi \geq 2d + 2 - \floor{\sqrt{8d+1}}_{odd},
  \end{equation}
  where $\floor{x}_{odd}$ denotes the greatest odd integer lower than $x$.
\end{proposition}

For a group $\G \subset SO(\R^3) \approx PSU(\C^2)$, consider
its preimage $2\G \subset SU(\C^2)$, i.e. the corresponding binary
group. Then one has the natural representations of $2\G$ on
every symmetric power $S^k\C^2$. Since $PSU(\C^2) = SU(\C^2)/\brc{\pm 1}$,
the representations of $2\G$ for $k = 2m$ are, in fact, mapped
through the representations of $\G$. Later on, it will be useful  to
understand the structure of some particular representations.

\begin{lemma}\label{lem:rep-info}
  Fix an element of the maximal order $\gamma_0 \in \G \subset PSU(\C^2)$ and
  choose a unitary basis
  $\C^2 = Span\brt{v_0,v_1}$ consisting of eigenvectors of
  $\gamma_0$. Then one has the following.
  \begin{itemize}
    \item If $\G = O$, then $S^4\C^2 = \mathbf{3}\oplus\mathbf{2}$,
    where $\mathbf{2}=Span\brt{v_0^4 + v_1^4, v_0^2v_1^2}$.
    The action of $T \subset O$ splits $\mathbf{2} = \mathbf{1}\oplus\mathbf{1'}$
    into two (non-isomorphic) $1$-dimensional representations leaving the same vector
    subspace $\mathbf{3}$.
    \item If $\G = I$, then $S^4\C^2 = \mathbf{5}$ and
    $S^6\C^2 = \mathbf{3}\oplus\mathbf{4}$. As a representation of
    $2I$, one has $S^7\C^2 = \mathbf{2}\oplus\mathbf{6}$, where
    $\mathbf{2} = Span\brt{v_0^7 - 7v_0^2v_1^5, v_1^7 + 7v_0^5v_1^2}$
    up to permuting $v_0$ and $v_1$.
  \end{itemize}
\end{lemma}
\begin{proof}
  To obtain the decomposition of $S^k\C^2$ into irreducible
  representations, we are going to use the McKay
  correspondence~\cite*{McKay:1980:correspondence}. Namely, consider
  the following graphs.
  \begin{itemize}
    \item For $2T$,
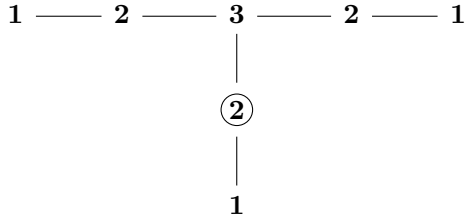
    \begin{equation}
      \begin{tikzcd}
        \mathbf{1} \arrow[r,dash] & \mathbf{2} \arrow[r,dash] & \mathbf{3} & \mathbf{2} \arrow[l,dash] & \mathbf{1} \arrow[l,dash] \\
        && \circled{$\mathbf{2}$} \arrow[u,dash] \\
        && \mathbf{1} \arrow[u,dash]
      \end{tikzcd}
    \end{equation}
    \item For $2O$,
    
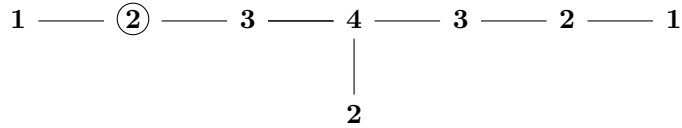
\begin{equation}
      \begin{tikzcd}
        \mathbf{1} \arrow[r,dash] & \circled{$\mathbf{2}$} \arrow[r,dash] & \mathbf{3} \arrow[r,dash] &
        \mathbf{4} \arrow[l,dash] & \mathbf{3} \arrow[l,dash] & \mathbf{2} \arrow[l,dash] & \mathbf{1} \arrow[l,dash]\\
        &&& \mathbf{2} \arrow[u,dash]
      \end{tikzcd}
    \end{equation}
    \item For $2I$,
    
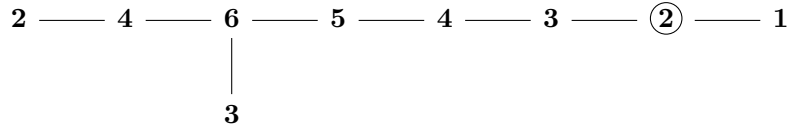
\begin{equation}
      \begin{tikzcd}
        \mathbf{2} \arrow[r,dash] & \mathbf{4} \arrow[r,dash] & \mathbf{6} &
          \mathbf{5} \arrow[l,dash] & \mathbf{4} \arrow[l,dash] & \mathbf{3} \arrow[l,dash]
          & \circled{$\mathbf{2}$} \arrow[l,dash] & \mathbf{1} \arrow[l,dash]\\
        && \mathbf{3} \arrow[u,dash]
      \end{tikzcd}
    \end{equation}
  \end{itemize}
  The set of vertices $V$ of the graphs consists of all the irreducible
  non-isomorphic representations of
  $2\G \subset SU(\C^2)$, where the numbers denote the dimensions of the representations.
  The vertex $\circled{$\mathbf{2}$}$ corresponds to the canonical representation on $\C^2$.
  The edges $E$ connect the vertices in a way that
  \begin{equation}
    \mathbf{x}\otimes\circled{$\mathbf{2}$} = \bigoplus_{(\mathbf{x}, \mathbf{y})\in E} \mathbf{y}
  \end{equation}
  for any $\mathbf{x} \in V$.

  The McKay graphs presented above will help us inductively decompose the representations
  $S^k\C^2$ by the following procedure. Suppose we know the decomposition of
  $S^{k-1}\C^2$. Then we use the graphs to compute $S^{k-1}\C^2 \otimes \C^2$.
  On the other hand, $S^{k-1}\C^2 \otimes \C^2 = S^{k}\C^2\oplus(S^{k}\C^2)^{\bot}$,
  where $\dim (S^{k}\C^2)^{\bot} = k-1$ and $\dim S^{k}\C^2 = k+1$. By
  simple dimension counting, we will be able to obtain the expression for $S^{k}\C^2$.

  Thus, for the representations of $O \approx S_4$, we obtain
  $S^3\C^2 = \mathbf{4}$ and $S^3\C^2\otimes \C^2 =
  \mathbf{3}\oplus\mathbf{2}\oplus\mathbf{3'}$, so
  $S^4\C^2 = \mathbf{3}\oplus\mathbf{2}$. To find
  $\mathbf{2} \subset S^4\C^2$, we
  proceed as follows. We may think that under a stereographic projection the element $\gamma_0$
  acts as $z \mapsto i z$. Another element $\gamma_1 \in O$ (the rotation around $y$ axis)
  is given by the formula $z \mapsto \frac{z - 1}{z + 1}$. Since $\mathbf{2}$
  is the unique $2$-dimensional representation of $O$, it is known that
  $\gamma_0|_{\mathbf{2}}$ has eigenvalues $\brc{\pm 1}$, see, e.g.,~\cite[\S~2.3]{Fulton-Harris:1991:representation-theory}.
  Therefore, $v_0^2v_1^2 \in \mathbf{2}$. To find another vector, we expand
  $\gamma_1(v_0^2v_1^2) = (v_0 + v_1)^2(v_1-v_0)^2$ and see that
  $v_0^4 + v_1^4 \in \mathbf{2}$ as well.

  For the representation of $T$ on the same $S^4\C^2$,
  we obtain $S^4\C^2 = \mathbf{3}\oplus\mathbf{1}\oplus\mathbf{1'}$.
  Since $T \subset O$, the irreducible representations of $O$ split
  into irreducible representations of $T$. Therefore, $\mathbf{2}=
  \mathbf{1}\oplus\mathbf{1'}$, where $\mathbf{2}$ is the representation
  of $O$ from above.

  When $\G = I$ and $2\G = 2I \approx SL_2(\mathbb{F}_5)$, we have $S^4\C^2 = \mathbf{5}$, $S^6\C^2 = \mathbf{3}\oplus\mathbf{4}$,
  and $S^7\C^2 = \mathbf{2}\oplus\mathbf{6}$. Consider the elements $\tilde{\gamma}_0,
  \tilde{\gamma}_1 \in 2I$,
  \begin{equation}
    \tilde{\gamma}_0 = \begin{pmatrix}
      \epsilon^{-2} & 0\\
      0 & \epsilon^{2}
    \end{pmatrix}
    ,\quad
    \tilde{\gamma}_1 = \frac{1}{\sqrt{5}}\begin{pmatrix}
      -(\epsilon-\epsilon^{-1}) & \epsilon^2-\epsilon^{-2}\\
      \epsilon^2-\epsilon^{-2} & \epsilon-\epsilon^{-1}
    \end{pmatrix}
  \end{equation} where
  $\epsilon = e^{\frac{2\pi}{5}i}$; see, e.g., (1.2.22) in~\cite*{Toth:2002:finite-mobius-groups}.
  The representation
  $\mathbf{2}$ is the composition of the canonical representation
  $\circled{$\mathbf{2}$}$ with an outer-automorphism interchanging the conjugacy
  classes
  $C(\tilde{\gamma}_0)$ and $C(\tilde{\gamma}_0^2)$, see, e.g.,~\cite[\S~5.2]{Fulton-Harris:1991:representation-theory}.
  Hence, $\tilde{\gamma}_0|_{\mathbf{2}}$ has the eigenvalues
  $\epsilon, \epsilon^{-1}$. As a consequence,
  \begin{equation}
    \mathbf{2} = Span
  \brt{\alpha_0 v_0^7 + \beta_0 v_0^2 v_1^5, \alpha_1 v_1^7 + \beta_1 v_1^2 v_0^5}
  \end{equation}
  for some $\alpha_i, \beta_i \in \C$. By looking
  at the term of $\tilde{\gamma}_1(\alpha_0 v_0^7 + \beta_0 v_0^2 v_1^5)$
  with $v_0^6v_1$ and using the invariance of the representation $\mathbf{2}$, we obtain
  \begin{multline}
    0 = 7\alpha_0(\epsilon^2-\epsilon^{-2})(\epsilon-\epsilon^{-1})^6
    \\+ \beta_0\brs{5(\epsilon-\epsilon^{-1})^3(\epsilon^2-\epsilon^{-2})^4
    - 2(\epsilon^2-\epsilon^{-2})^6(\epsilon-\epsilon^{-1})}.
  \end{multline}
  So,
  \begin{equation}
    \frac{\beta_0}{\alpha_0} = -\frac{7}
    {5\brr{\frac{\epsilon^2-\epsilon^{-2}}{\epsilon-\epsilon^{-1}}}^3
    - 2\brr{\frac{\epsilon^2-\epsilon^{-2}}{\epsilon-\epsilon^{-1}}}^5}
    = -7.
  \end{equation}
  In a similar way, by looking
  at the term of $\tilde{\gamma}_1(\alpha_1 v_1^7 + \beta_1 v_1^2 v_0^5)$
  with $v_1^6v_0$, we deduce that $\frac{\beta_1}{\alpha_1} = 7$.
\end{proof}

For $\Lambda_k(\Sph^2;\G)$ with $k\geq 4$, one needs to study
the space of $\G$-equivariant harmonic maps in higher dimensional
spheres. The simplest case is to look for
$\G$-equivariant $\Phi\in Harm_{d_0}^f(\Sph^2,\Sph^{2m})$,
where $d_0 = \frac{1}{2}m(m+1)$.
The map $\Phi_0\in Harm_{d_0}^f(\Sph^2,\Sph^{2m})$ given by spherical
harmonic is certainly an example of a $\G$-equivariant map
for any group $\G \subset O(\R^3)$ since spherical harmonics induce a
metric proportional to the round one.
\begin{remark}
  Because of~\eqref{eq:spher-harmonics}, one has
  $\nul \Phi_0 = \dim \ker (\Delta_{\Sph^2} - \abs{d\Phi_0}^2_{g_{\Sph^2}}) = 2m+1,$
  and the index is $\ind \Phi_0 = \sum_{i = 0}^{m-1}(2i+1) = m^2$.
  At the same time, given $\gamma \in O(\R^3)$, one can see that
  $\gamma^* \Phi^i_0 \in \ker (\Delta_{\Sph^2} - \abs{d\Phi_0}^2_{g_{\Sph^2}})$,
  so $\gamma^* \Phi^i_0$ is a linear combination of the coordinate functions
  $\brc{\Phi^i_0}$.
\end{remark}
Denote by $r_0\colon O(\R^3) \to O(\mathbb{R}^{2m+1})$ the representation
induced by $\Phi_0$.
\begin{proposition}
  Let $\G \subset SO(\R^3)$ be a finite group and $r_0$ be as above. Then
  the restricted representation $r_0\colon \G \to SO(\mathbb{R}^{2m+1})$
  is isomorphic to the representation
  on the $2m$-th symmetric power $S^{2m}\mathbb{C}^2 \approx \mathbb{C}^{2m+1}$ via
  the isomorphism $SO(\R^3) \approx
  PSU(\mathbb{C}^2)$ through a stereographic projection.

  Isometrically different classes $[\Phi] \subset Harm_{d_0}^f(\Sph^2,\Sph^{2m})$
  of $\G$-equivariant harmonic maps
  are in one-to-one correspondence with (Hermitian) self-adjoint
  operators of the form $D \in SO(\mathbb{C}^{2m+1})$ such that
  $D \in \operatorname{End}(r_0)$ and $D = \bar{A}^t A$ for some
  $A \in SO(\mathbb{C}^{2m+1})$.
\end{proposition}
\begin{proof}
  For an arbitrary harmonic map $\Phi \in Harm_{d_0}^f(\Sph^2,\Sph^{2m})$
  and $\gamma \in \G$, it is clear that $\gamma^* \Phi$, is another harmonic map. If
  $\xi$ is the directrix curve of $\Phi$, it follows that
  $\gamma^* \xi$ is the directrix curve of $\gamma^* \Phi$. Then
  Propositions~\ref{prop:min-degree-maps} imply that $\gamma^* \xi$
  can also be expressed as $\gamma^* \xi = r(\gamma) \xi$ for some
  $r(\gamma) \in SO(\mathbb{C}^{2m+1})$. Thus, every
  $\xi$ gives rise to a (right) representation
  \begin{equation}\label{eq:directrix-rep}
    r_{\xi}\colon \G \to SO(\mathbb{C}^{2m+1}).
  \end{equation}
  Now, let $\xi_0$ be the directrix curve of the map $\Phi_0$ given by spherical
  harmonics. Then $r_{\xi_0}(\gamma) = r_0(\gamma)$ since $r_0$ is real and
  $r_0(\gamma) \Phi_0$ is proportional to $(\wedge^m r_0(\gamma))\xi_m \wedge \overline{(\wedge^m r_0(\gamma))\xi_m}$.

  To prove the first statement of the proposition
  we proceed as follows.
  The group $\G$ acts on $\Sph^2$, which is isomorphic
  to $\mathbb{CP}^1$ via a stereographic projection. Under this isomorphism,
  $\G$ becomes a subgroup of $PSU(\mathbb{C}^2) \approx
  \quot{SU(\mathbb{C}^2)}{\brc{\pm 1}}$. So, the action of
  $\G$ is well defined on even symmetric powers
  $S^{2m}\mathbb{C}^2$.
  Given a linearly full holomorphic curve $\xi$ of degree $2m$, one can view it as
  the projectivization of an invertible operator
  $\xi \in hom(S^{2m}\mathbb{C}^2,\mathbb{C}^{2m+1})$. The
  action of $\G$ turns into the multiplication to the right,
  $\gamma^* \xi = \xi r'(\gamma)$. The degree
  of the directrix curves of harmonic maps
  in $Harm_{d_0}^f(\Sph^2, \Sph^{2m})$ is exactly $2m$
  (see~\cite[Proposition~6.15]{Barbosa:1975:min-imersions}).
  Therefore, after recalling the definition of $r_\xi$
  in~\eqref{eq:directrix-rep}, one obtains
  \begin{equation}
    r_\xi = \xi r' \xi^{-1}.
  \end{equation}

  To prove the second statement,
  note that $r_{\xi} = A r_{\xi_0} A^{-1}$ since $\xi  = A \xi_0$ for some
  $A \in SO(\mathbb{C}^{2m+1})$ by Proposition~\ref{prop:min-degree-maps}.
  One sees from Proposition~\ref{prop:curve-ridgidity} that $\G$-invariant
  metrics induced by $\Phi \in Harm_{d_0}(\Sph^2,\Sph^{2m})$ correspond to
  unitary representations $r_{\xi}$, i.e.
  \begin{equation}\label{eq:unitary-real}
    r_{\xi}\colon \G \to SO(\mathbb{C}^{2m+1}) \cap U(\mathbb{C}^{2m+1}) = SO(\mathbb{R}^{2m+1}).
  \end{equation}
  It follows from the same propositions that
  isometrically different harmonic maps
  correspond to the right cosets
  \begin{equation}
    SO(\mathbb{R}^{2m+1}) \cdot A, \quad \text{where } A \in SO(\mathbb{C}^{2m+1}),
  \end{equation}
  i.e. the directrix curves
  $BA\xi_0$ and $A \xi_0$ are isometric if and only if $B \in SO(\mathbb{R}^{2m+1})$.
  Therefore, all the unitary representations are among those of the form
  $Ar_0A^{-1}$. If we set $\brt{x,y}_A := \brt{Ax,Ay}$, where
  $\brt{\cdot, \cdot}$ is the standard Hermitian form on
  $\mathbb{C}^{2m+1}$, we see that for a unitary representation
  $Ar_0A^{-1}$, one has
  \begin{equation}
    \brt{r_0(\gamma)x,y}_A = \brt{Ar_0(\gamma)A^{-1}Ax,Ay} =
    \brt{Ax,Ay} = \brt{x,y}_A
  \end{equation}
  for all $x,y \in \mathbb{C}^{2m+1}$, and vice versa, if $r_0$
  preserves $\brt{\cdot,\cdot}_A$, the representation $Ar_0A^{-1}$ is unitary.
  Notice that the map $A \mapsto D := \bar{A}^t A$
  is one-to-one on the (right) cosets $A' \in
  SO(\mathbb{R}^{2m+1})\cdot A$ and
  $D \in SO(\mathbb{C}^{2m+1})$. In fact,
  \begin{equation}
    (\bar{A}^t A)^t (\bar{A}^t A)= A^t(\bar{A}\bar{A}^t) A = \Id.
  \end{equation}
  Because the forms $\brt{\cdot, \cdot}$, and $\brt{\cdot, \cdot}_A$
  are preserved by $r_0$, one can check that the self-adjoint
  operator $D$ is an endomorphism of the representation
  $r_0$.
\end{proof}
\begin{corollary}\label{cor:no-new-metrics-S4}
  Let $(\G, m) \in \brc{T, O}\times \brc{2} \cup \brc{I}\times \brc{2,3}$. Then
  up to a conformal automorphism of $\Sph^2$,
  the round metric is the only $\G$-invariant metric coming from
  $Harm_{d_0}^f(\Sph^2,\Sph^{2m})$.
\end{corollary}
\begin{proof}
  From Lemma~\ref{lem:rep-info}, we see that in all the cases,
  the representation $r_0\colon\Gamma \to SO(\mathbb{R}^{2m+1})
  \subset SO(\mathbb{C}^{2m+1})$
  decomposes $r_0 = \oplus_{i} r_i$
  into the direct sum of non-isomorphic irreducible (real) representations
  $r_i$. By Schur's lemma,
  \begin{equation}
    D \in hom(r_0,r_0) = \oplus_{i} hom(r_i,r_i) =
  Span \brc{\Id_i \colon r_i \to r_i}.
  \end{equation}
  So, $D = diag(\cdots d_i \cdots)$, and $d_i > 0$ since $D$ is self-adjoint.
  Then $\Id = D^t D = D^2 $ implies $D = \Id$.
\end{proof}

One also needs to study harmonic maps
$\Phi\colon\Sph^2\to\Sph^2$. Since such maps are weakly conformal, they are
holomorphic
if we choose the appropriate orientation on $\Sph^2$. Then
the harmonic degree is simply the degree of $\Phi$ considered
as a holomorphic map.

\begin{lemma}\label{lem:sph-corner-cases}
  Let $\G \in \brc{T,O}$ and $\Phi_1\colon \Sph^2\to \mathbb{CP}^1$ be a
  $\G$-equivariant holomorphic map of degree $4$. Then $\ind \Phi_1 = 7$
  and up to a unitary transformation
  of $\mathbb{CP}^1$, $\Phi_1 = [z^4 + 1 : \sqrt{2}z^2]$.

  Let $\Phi_2\colon \Sph^2\to \mathbb{CP}^1$ be a
  $I$-equivariant holomorphic map of degree $7$. Then $\ind \Phi_2 = 13$ and
  up to a unitary transformation of $\mathbb{CP}^1$,
  $\Phi_2 = [z^7 - 7z^2 : 7z^5 + 1]$.
\end{lemma}
\begin{proof}
  If $\Phi\colon\Sph^2\to\mathbb{CP}^1$ is
  a $\G$-equivariant map, by Proposition~\ref{prop:curve-ridgidity}, one has
  $\gamma^*\Phi = r(\gamma)\Phi$,
  where $r\colon\G\to PSU(\C^2)$. On the other hand,
  $\phi$ is the projectivization of a polynomial. So, up to multiplication
  by a constant, one can view $\Phi$ as an element of $hom(S^d\C^2, \C^2)$.
  Therefore,
  \begin{equation}
    \Phi \in hom(S^d\C^2, \tilde{r})
  \end{equation}
  is a homomorphism of representations of $2\G$, where $\tilde{r}$
  is induced by $r$.

  \textbf{The case $\mathbf{\G \in \brc{T,O}}$.} From Lemma~\ref{lem:rep-info},
  one has $S^4\C^2 = \mathbf{3}\oplus\mathbf{2}$ for $\G = O$, and
  $\mathbf{2} = \mathbf{1}\oplus\mathbf{1'}$ for $\G = T$.
  Then Schur's lemma tells us that
  $\tilde{r}\approx\mathbf{2}$ or $\tilde{r}\approx\mathbf{1}\oplus\mathbf{1'}$.
  Fix a Hermitian basis $v_0, v_1 \in \C^2$. Let $z_0, z_1 \in
  (\C^2)^*$ be the coordinates of a point in that basis and $z = \frac{z_0}{z_1}$ be
  the coordinate on $\mathbb{CP}^1$.
  Then up to a unitary transformation, the holomorphic map $\Phi$ that has $\mathbf{3}$
  as its kernel and unitarily maps $\mathbf{2}$ (or $\mathbf{1}\oplus\mathbf{1'}$)
  onto $\C^2$ has the form
  \begin{equation}
    \Phi(z) = [z^4+1 : \sqrt{2}z^2].
  \end{equation}
  since $\mathbf{2} = \mathbf{1}\oplus\mathbf{1'} = Span\brt{v_0^4 + v_1^4, v_0^2v_1^2}$.

  To calculate the index, we will prove that
  $\nul \Phi = 3$. Then by~\cite[Theorem~23]{Montiel-Ros:1991:s2-harm-maps},
  we will have $\ind \Phi = 2d - 1$.

  Any holomorphic map $\Phi\colon \Sph^2 \to\mathbb{CP}^1$
  of degree $d$ is given by a pair polynomials
  $\Phi(z) = [p(z):q(z)]$ without common roots such that
  $\max \brc{\deg p, \deg q} = d$. Thus, up to a linear transformation
  of $\mathbb{CP}^1$, the space of
  all such holomorphic maps is parametrized by a certain
  open set of the Grassmannian $G(2, d+1)$. For a pair of polynomials of
  degree $\leq d$,
  define the following map,
  \begin{equation}
    \mathcal{R}(p,q) = p'(z)q(z) - p(z)q'(z).
  \end{equation}
  One has $\deg \mathcal{R}(p,q) \leq 2d - 2$, and the zeros of $\mathcal{R}(p,q)$
  correspond to the ramification points
  of $\Phi$. Since $\mathcal{R}$ is bilinear
  and antisymmetric, the map $\Psi\colon G(2, d+1) \to \mathbb{CP}^{2d-2}$
  given by
  \begin{equation}
    \Psi([p\wedge q]) = [\mathcal{R}(p,q)]
  \end{equation}
  is correctly defined.  Theorem~B of~\cite{Ejiri:1994:nonfull-min-surface-in-s4}
  then implies that
  \begin{equation}
    \nul \Phi = 3 + 2\dim_{\C} \ker d_{[p\wedge q]}\Psi.
  \end{equation}

  In our case, we may set $p(z) = z^4 + 1$ and $q(z) = z^2$,
  and let us choose the coordinate chart $U \subset G(2,5)$,
  \begin{equation}
    U  = \set*{(p+ \delta p)\wedge (q + \delta q)}{\delta p, \delta q
  \in Span\brt{1, z, z^3}}.
  \end{equation}
  Since
  \begin{equation}
    \mathcal{R}(p+ \delta p,\ q + \delta q) = \mathcal{R}(p,q) +
    \mathcal{R}(p,\delta q) - \mathcal{R}(q,\delta p)
    + \mathcal{R}(\delta p, \delta q),
  \end{equation}
  we see that
  \begin{equation}
    d_{[p\wedge q]}\Psi(\delta p, \delta q) =
    d\pi[\mathcal{R}(p,\delta q) - \mathcal{R}(q,\delta p)],
  \end{equation}
  where $\pi\colon \C^{2d-1} \to \mathbb{CP}^{2d-2}$ is the canonical projection.
  Since
  \begin{align}
    \mathcal{R}(z^4+1, z^2) &= 4z^5 - 2z(z^4+1) = 2(z^5-z),&
    \\ \mathcal{R}(z^4+1, z^k) &= (4-k)z^{k+3} - kz^{k-1},
    &\quad k \in \brc{0,1,3}
    \\ \mathcal{R}(z^2, z^k) &= (2-k)z^{k+1},
    &\quad k \in \brc{0,1,3}
  \end{align}
  are linearly independent, we conclude that $\ker d_{[p\wedge q]}\Psi = \brc{0}$.

  \textbf{The case $\mathbf{\G = I}$.}
  From Lemma~\ref{lem:rep-info}, $S^7\C^2 = \mathbf{2}\oplus\mathbf{6}$ as the representation of $2I$,
  so there exists only one $2$-dimensional $I$-invariant subspace, and it
  is expressed as $Span\brt{v_0^7 - 7v_0^2v_1^5, v_1^7 + 7v_0^5v_1^2}$.
  Hence, $\Phi_2 = [z^7 - 7z^2 : 7z^5 + 1]$, and we choose the coordinate
  chart $U \subset G(2, 8)$,
  \begin{equation}
    U  = \set*{(p+ \delta p)\wedge (q + \delta q)}{\delta p, \delta q
  \in Span\brt{1, z, z^2, z^3, z^4, z^6}}.
  \end{equation}
  Again, the polynomials
  \begin{align}
    \mathcal{R}(z^7 - 7z^2,\ 7z^5 + 1) &=14(z^{11} + 11z^6 - z),&
    \\ \mathcal{R}(z^7 - 7z^2,\ z^k) &= (7-k)z^{k+6} - 7(2-k)z^{k+1},
    &\quad 0\leq k \leq 6,\ k\neq 5
    \\ \mathcal{R}(7z^5 + 1,\ z^k) &= 7(5-k)z^{k+4} - kz^{k-1},
    &\quad 0\leq k \leq 6,\ k\neq 5
  \end{align}
  are linearly independent, so $\ker d_{[p\wedge q]}\Psi = \brc{0}$.
  \begin{remark}
    Note that $[z^5-z]$ and $[z^{11} + 11z^6 - z]$ are unique lines of polynomials
    whose roots are the vertices of the (fixed) octahedron and icosahedron correspondingly,
    and these vertices are the branching points of $\Phi$.
  \end{remark}
\end{proof}

\subsubsection{Proof of Theorem~\ref{thm:equiv-eigenval-sphere}}
  The group $\G \in \brc{T, T_h, T_d, O, O_h, I, I_h}$ is the symmetry
  group of the corresponding Platonic solid $P$, or in the case of
  $\G = T_h$, we set $P$ to be a pyritohedron.

  One can think that $P$ is inscribed in
  $\mathbb{S}^2$. Then the rays emanating from the origin project
  $P$ onto $\mathbb{S}^2$ so that every point of $\mathbb{S}^2$
  belongs to a image of a vertex, an edge or a face. The group
  $\G$ acts transitively on the vertices, edges and faces, except
  for the case of a pyritohedron where vertices and edges
  split up into two types orbits of $(8,12)$ and $(6,24)$ elements
  correspondingly.
  Thus, one has
  \begin{itemize}
    \item $\G \in \brc{T, T_d}$,
    \begin{equation}\label{eq:N_G-thetrahedron}
      \mathbb{N}_\G = \set*{4\alpha+6\beta}{\alpha,\beta \in \mathbb{N}};
    \end{equation}
    \item $\G \in \brc{T_h, O, O_h}$,
    \begin{equation}
      \mathbb{N}_\G = \set*{6\alpha+8\beta}{\alpha,\beta \in \mathbb{N}};
    \end{equation}
    \item $\G \in \brc{I, I_h}$,
    \begin{equation}
      \mathbb{N}_\G = \set*{12\alpha+20\beta+30\gamma}{\alpha,\beta,\gamma \in \mathbb{N}}.
    \end{equation}
  \end{itemize}

  Notice that
  any maximizing $L^\infty$-measure $\mu$, i.e.
  $\overline{\lambda}_k(\mu) = \Lambda_k(\Sph^2;\G)$, has the form
  $\abs{d\Phi}^2_{g} \dv{g}$ for some $\Phi \in Harm_d^f(\Sph^2, \Sph^{2m})$
  as it follows from the proofs of Proposition~\ref{prop:max-sequence} and
  Lemma~\ref{lem:regularity-in-good-pt}. Then
  \begin{equation}
    \Phi^* g_{\Sph^n} = \frac{1}{2}\abs{d\Phi}^2_{g} g,
    \quad
    \lambda_k(\abs{d\Phi}^2_{g} g) = 1
  \end{equation}
  for $g = g_{\Sph^2}$ and by Proposition~\ref{prop:harm-sheres},
  \begin{equation}
    \Lambda_k(\mathbb{S}^2;\G) = \overline{\lambda}_k(\abs{d\Phi}^2_{g})
    = 1 \cdot \int \abs{d\Phi}^2_{g} dv_g = 8 \pi d.
  \end{equation}
  Since $\lambda_k(\abs{d\Phi}^2_{g} g) = 1$, by the definition of
  $\ind \Phi$, one also has
  \begin{equation}\label{eq:sph-max-equality}
    \ind \Phi \leq k.
  \end{equation}

  \paragraph{The case $\mathbf{\G \in \brc{T, T_d}}$.}
  The expression~\eqref{eq:N_G-thetrahedron} of $\mathbb{N}_\G$ implies that
  $k \in \mathbb{N}_\G \cup (\mathbb{N}_\G + 1)$ for any $k\geq 4$. That
  means $\Lambda_k(\Sph^2;\G) = 8\pi k = \Lambda_k(\Sph^2)$ and
  there is no maximizing metric since otherwise, Proposition~\ref{prop:harm-sheres}
  would imply the existence of $\Phi \in Harm_d^f(\Sph^2,\Sph^{2m})$ such that
  \begin{equation}
    8\pi k = \Lambda_k(\Sph^2;\G) = 8\pi d.
  \end{equation}
  On the other hand, if $d \geq 2$,
  one always has $\floor{\sqrt{8d+1}}_{odd} \leq d+1$. Hence,
  $\ind \Phi \geq d+1$ by Proposition~\ref{prop:ind-est},
  which leads to the contradiction $k = d < d+1 \leq k$.

  To deal with $k \leq 3$, we apply
  Corollary~\ref{cor:non-babl} to obtain a map
  $\Phi \in Harm_d^{f}(\Sph^2, \Sph^{2m})$ satisfying~\eqref{eq:sph-max-equality}.
  Then by Proposition~\ref{prop:ind-est}, one has $d \leq 2$, but
  the case $d= 2$ is excluded by Lemma~\ref{lem:equiv-degree}. So, $\Phi$ is an automorphism
  of the sphere. The $\G$-equivariance of $\Phi$ and Schur's lemma means that
  it is isomorphism of the representations of $\G$, whence it is not hard
  to conclude that $\Phi \in O(\R^3)$.

  \paragraph{The case $\mathbf{\G \in \brc{T_h, O, O_h}}$.}
  As in the previous case, $k \in \mathbb{N}_\G \cup (\mathbb{N}_\G + 1)$
  when $k \in [6,9] \cup [12, \infty)$, so $\Lambda_k(\Sph^2;\G) = 8\pi k$.
  For $1\leq k \leq 5$, Corollary~\ref{cor:non-babl} produces a map
  $\Phi \in Harm_d^{f}(\Sph^2, \Sph^{2m})$ satisfying~\eqref{eq:sph-max-equality}.
  Therefore, $\ind \Phi \leq 5$ and
  by Proposition~\ref{prop:ind-est}, we obtain $d \leq 5$.
  Using Lemma~\ref{lem:equiv-degree}, Proposition~\ref{prop:ind-est},
  and the inequality~\eqref{eq:sph-max-equality},
  we see that
  \begin{itemize}
    \item $(m,d) = (1,1)$ for $1\leq k \leq 3$,
    \item $(m,d) = \brc{(1,1), (2,3)}$ for $k  = 4$,
    \item $(m,d) \in \brc{(1,1), (2,3), (1, 4)}$ for $k=5$.
  \end{itemize}
  The possibility $(1, 4)$ is excluded by Lemma~\ref{lem:sph-corner-cases},
  so the maximal value of $d$ is achieved on $(m,d) = (2,3)$.
  Then Corollary~\ref{cor:no-new-metrics-S4} applied
  to $\G_+$ implies
  that there are no new metric besides the round one.

  For $k \in \brc{10,11}$, there is no maximizing metric. Otherwise,
  $\ind \Phi \leq 11$ and $\frac{1}{8\pi} \Lambda_k(\Sph^2;\G) =
  d \leq 8$ by Proposition~\ref{prop:ind-est}. But we can do better
  on a disjoint union of $9 = 8 + 1$ (or $7 = 6 + 1$) spheres,
  so Theorem~\ref{thm:main} implies
  \begin{equation}
    \Lambda_k(\Sph^2;\G) = 8\pi \cdot 9
    = 8\,\overline{\lambda}_1(g_{\Sph^2}) + \overline{\lambda}_{k-8}(g_{\Sph^2})
    = 6\,\overline{\lambda}_1(g_{\Sph^2}) + \overline{\lambda}_{k-6}(g_{\Sph^2}).
  \end{equation}

  \paragraph{The case $\mathbf{\G \in \brc{I, I_h}}$.}
  Since $60 = \operatorname{lcm}(12,20,30)$, it is clear that
  $k \in \mathbb{N}_\G \cup (\mathbb{N}_\G + 1)$ when $k \geq 60$.
  For $1 \leq k \leq 11$, by Corollary~\ref{cor:non-babl},
  there exists a maximizer $\Phi$ with
  $\ind \Phi \leq 11$, which leads to $d \leq 8$ by Proposition~\ref{prop:ind-est}.
  Using Lemma~\ref{lem:equiv-degree} and the index estimates again, we see that
  \begin{itemize}
    \item $(m,d) = (1,1)$ for $1\leq k \leq 3$,
    \item $(m,d) = \brc{(1,1), (2,3)}$ for $4 \leq k \leq 8$,
    \item $(m,d) \in \brc{(1,1), (2,3), (3,6), (1, 7)}$ for $k \in \brc{9,10,11}$.
  \end{itemize}
  The possibility $(1, 7)$ is excluded by Lemma~\ref{lem:sph-corner-cases}, and
  Corollary~\ref{cor:no-new-metrics-S4} tells us
  that there are no new metric besides the round one.

  For $13 \leq k$, there is no maximizing metric. Otherwise,
  $\frac{1}{8\pi} \Lambda_k(\Sph^2;\G) = d \geq 13$ since we can
  always attach at least $12$ spheres (bubbles). Then one verifies that
  $\floor{\sqrt{8d + 1}}_{odd} \leq d - 4$, and Proposition~\ref{prop:ind-est}
  implies
  \begin{equation}
    k \leq 2d + 2 - d + 4 = d + 6.
  \end{equation}
  On the other hand,
  we can always decompose $k = k' + s$, where $s \leq 7$,
  $k' = \max(\mathbb{N}_\G \cap [0, k])$, and attach $k$
  bubbles resulting in the estimate
  \begin{equation}
    \frac{\Lambda_k(\Sph^2;\G)}{8\pi}  = d \leq k' + s - 6
    < k' + \frac{\overline{\lambda}_s(g_{\Sph^2})}{8\pi}
  \end{equation}
  since $\max_{s\leq 7}(s - \frac{1}{8\pi}\overline{\lambda}_s(g_{\Sph^2}))
   = 4 < 6$. Thus, the disjoint union of $k'+1$ spheres produce
   a greater value of $\overline{\lambda}_k$.

  We prove by induction that $\Lambda_k(\Sph^2;\G) = 8\pi k' + \overline{\lambda}_{k-k'}(g_{\Sph^2})$.
  We suppose that $k \notin \mathbb{N}_\G \cup (\mathbb{N}_\G + 1)$, or else
  $\Lambda_{k}(\Sph^2;\G) = 8\pi k$ and the statement is obvious.
  Applying the induction assumption to Theorem~\ref{thm:main}, one obtains
  $b\in\mathbb{N}_\G \cap [k - 11,\ k]$ such that
  \begin{equation}
    \Lambda_{k}(\Sph^2;\G) = 8\pi b + \Lambda_{k-b}(\Sph^2; \G).
  \end{equation}
  If $b < k'$, we see that $k' - b \geq 4$ since $k \notin \mathbb{N}_\G \cup (\mathbb{N}_\G + 1)$
  and $\mathbb{N}_\G$ consists of even numbers. Then we arrive at the contradiction
  $8\pi k' + \overline{\lambda}_{k-k'}(g_{\Sph^2}) \leq 8\pi b + \overline{\lambda}_{k-b}(g_{\Sph^2})$
  since
  \begin{equation}
    k'-b \leq \frac{1}{8\pi} \max_{\substack{k-b\leq 11 \\ k-k' \leq 11}}
    \brc{\overline{\lambda}_{k-b}(g_{\Sph^2})
      - \overline{\lambda}_{k-k'}(g_{\Sph^2})} = 6 - 3 < 4.
  \end{equation}

\newpage
\appendix\section{Bubbling analysis}\label{app:blow-up}

Let $\brr{M,g}$ be a compact Riemannian manifold possibly with
a nonempty boundary $\partial M$ and $\dim M = m \geq 2$.
Denote by $\mathcal{M}^+_c(M)$ the space of all (positive) continuous Radon measures
on $M$. Consider a sequence $\brc{\mu_\varepsilon} \subset \mathcal{M}^+_c(M)$
such that $\mu_\varepsilon \oset{w^*}{\to}\mu$. Notice that $\mu$ is decomposable into its continuous and atomic parts,
\begin{equation}
  \mu = \mu^c + \sum_{p \in A(\mu)} w_p \delta_{p},
\end{equation}
where $A(\mu) = \brc{p \in M \colon \mu(\brc{p}) > 0}$ and $w_p  = \mu(\brc{p})$.

\begin{definition}\label{def:bubble-tree}
  Fix $\delta > 0$. We say that a converging sequence
  $\brc{\mu_\varepsilon} \subset \mathcal{M}^+_c(M)$,
  $\mu_\varepsilon \oset{w^*}{\to}\mu$,
  has a $\delta$-\emph{bubbling property} if up to a subsequence,
  there exists
  a finite collection of maps
  $T^i_\epsilon \colon U_i \to \R^m$, where
  $U_i \subset M$ is a coordinate chart around $p_i \in A(\mu)$
  with the following properties.
  \begin{enumerate}
    \item The maps $T^i_\epsilon$ have the form
      \begin{equation}
        T^i_\epsilon(x) = \frac{x - p^i_\epsilon}{\alpha^i_\epsilon},
        \quad\text{ where } p^i_\epsilon \to p_i, \ \alpha^i_\epsilon \searrow 0.
      \end{equation}
    \item Either for any $R > 0$ and small enough $\epsilon$, the balls
      $B(\alpha^i_\epsilon R,p^i_\epsilon)$ and
      $B(\alpha^j_\epsilon R,p^j_\epsilon)$ do not intersect,
      or one of them is contained in the other, say
      $B(\alpha^i_\epsilon R,p^i_\epsilon) \subset B(\alpha^j_\epsilon R,p^j_\epsilon)$ and
      $(T^j_\epsilon)^{-1}[B(\alpha^i_\epsilon R,p^i_\epsilon)]$ converges to a point in $\R^m$,
      in other words,
      \begin{equation}\label{eq:dif-rate-condition}
        \frac{\operatorname{dist}(p^i_\epsilon,p^j_\epsilon)}{\alpha^i_\epsilon + \alpha^j_\epsilon}
        +\frac{\alpha^i_\epsilon}{\alpha^j_\epsilon}
        +\frac{\alpha^j_\epsilon}{\alpha^i_\epsilon} \to \infty.
      \end{equation}
    \item $(T^i_\epsilon)_* [\mu_\epsilon] \overset{w^*}{\to} \mu_i \in \mathcal{M}^+(\R^m)$ and
      \begin{equation}\label{ineq:alomst-no-loss}
        (1-\delta)\mu(M) \leq \mu^c(M) + \sum_i \mu_i^c(\R^m) \leq  \mu(M).
      \end{equation}
  \end{enumerate}
  Here the convergence $(T_\epsilon)_* [\mu_\epsilon] \overset{w^*}{\to}
  \mu$ means
  \begin{equation}
    \int_{\R^m} \phi(x') d\mu(x') = \int_M \phi(\alpha_\epsilon x + p_\epsilon)
    d\mu_\epsilon(x),\quad \forall \phi \in C_0(\R^m).
  \end{equation}
\end{definition}

\begin{remark}\label{rem:bubbling-cons}
  The limiting measure decomposition in Definition~\ref{def:bubble-tree}
  is similar to the
  one discussed in~\cite{Petrides:2022:max-eigenval-higher-dim}.
  The proof that any sequence of measures with bounded eigenvalues,
  $\lambda_k(\mu_\epsilon)\geq c> 0$,
  exhibits the \emph{$\delta$-bubbling property} (see Proposition~\ref{prop:bubble-tree-existence})
  appears to be simpler than analogous constructions
  in~\cite{Karpukhin-Nadirashvili-Penskoi-Polterovich:2022:existence,
  Petrides:2018:exist-of-max-eigenval-on-surfaces,
  Petrides:2019:max-steklov-eigenval-on-surfaces,
  Karpukhin-Stern:2024:new-character-of-conf-eigenval}.
\end{remark}

\subsection{Bubbling construction}
\label{sec:bubl-tree-cons}
Let $\brc{\mu_\varepsilon}\subset \mathcal{M}^+_c(M)$ such that
$\mu_\varepsilon \overset{w^*}{\to}\mu$. Fix $\delta > 0$ and consider a point $p \in A(\mu)$,
$w := w_p$.
\begin{claim}
  Up to a subsequence of $\brc{\mu_\varepsilon}$, one can choose
  balls $B(r^2_\epsilon) \subset
  B(r_\epsilon)$ around $p$ such that
  \begin{equation}\label{eq:ball-choices}
    w - \varrho_\epsilon < \mu_\varepsilon[B(r^2_\epsilon)]
    \le \mu_\varepsilon[B(r_\epsilon)] < w + \varrho_\epsilon,
    \quad r_\epsilon, \varrho_\epsilon \to 0.
  \end{equation}
\end{claim}
\begin{proof}
  Fix $\varrho > 0$ and let us choose $r > 0$ small enough so that
  $w \leq \mu[B(2r)] < w + \varrho$. By the properties of
  the weak$^*$ convergence, one has
  \begin{equation}
    \mu[B(r^2)] \leq \liminf \mu_\epsilon\brs{B(r^2)}
    \quad\text{and}\quad
    \limsup \mu_\epsilon[\overline{B(r)}]
      \leq \mu[\overline{B(r)}].
  \end{equation}
  So, for a given $\varrho > 0$, there exists a measure $\mu_\epsilon$ satisfying
  the statement of the claim.
\end{proof}

We are going to work in a
coordinate chart $U \subset M$ centered around $p$, where
$U \subset \mathbb{R}^m$
or $\mathbb{R}^m_{+}$ if $p \in \partial M$.
Consider the inverse $\pi \colon \mathbb{R}^m \to
\mathbb{S}^m \setminus \brc{N}$ of the stereographic projection
and define the map
\begin{equation}
  \Gamma_\varepsilon \colon {(0, 1]\times\mathbb{S}^m\setminus\brc{N}} \to \mathbb{B}^{m+1},
\end{equation}
\begin{equation}
  \Gamma_\varepsilon (\alpha, \tilde{q}) = \frac{1}{\mu_\varepsilon[B(r_\epsilon)]}
  \int_{B(r_\epsilon)} \pi \brr{\frac{x - q}{\alpha}} d\mu_\varepsilon(x),
\end{equation}
under the identification $\tilde{q} = \pi(q)$.
By using the Lebesgue dominated convergence theorem, it is not hard to see that $\Gamma_\varepsilon$
extends to a continuous map $\Gamma_\varepsilon \colon {[0, 1] \times \mathbb{S}^m}
  \to \mathbb{B}^{m+1}$
and
\begin{equation}
  \abs{\Gamma_\varepsilon(1, \tilde{q}) - \tilde{q}} \le C r_\epsilon \le \frac{1}{2}
\end{equation}
  for small enough $\varepsilon$ because of our choice of the balls in~\eqref{eq:ball-choices}. It means
$\Gamma_\varepsilon|_{\brc{1}\times \mathbb{S}^m}$ is homotopic to the identity map
$\Id_{\mathbb{S}^m}$ and, therefore, has a non-zero degree.
At the same time, $\Gamma_\varepsilon|_{\brc{0}\times \mathbb{S}^m}$
maps everything to the north pole $\brc{N}$.
By topological reasons, $\Gamma_\varepsilon$ must be surjective on
$\Int \mathbb{B}^{m+1}$. In particular, there exists
a point $\brr{\alpha_\varepsilon, p_\varepsilon} \in
  {(0, 1] \times \mathbb{R}^m}$ such that
\begin{equation}
  \Gamma_\varepsilon\brr{\alpha_\varepsilon, \tilde{p}_\varepsilon}
  = (0, \delta - 1) \in \R^m\times\R,
\end{equation}
i.e. close enough to the south pole $(0,-1)$. Let us set
 $T^p_\epsilon \colon U \to \R^m$,
\begin{equation}
  T^p_\epsilon(x) = \frac{x - p_\epsilon}{\alpha_\epsilon},
\end{equation}
for some neighborhood $U$ of $p$.
Then it follows that the measures $\mu_{p,\varepsilon} \in \mathcal{M}^+_c(\mathbb{S}^m)$,
\begin{equation}
  \mu_{p,\varepsilon} :=
  (\pi\circ T^p_\epsilon)_*[\mu_\varepsilon \mres {B(r_\varepsilon)}],
\end{equation}
satisfy the equation
\begin{equation}\label{eq:balance}
  \frac{1}{\mu_{p,\epsilon}(\Sph^m)}\int_{\mathbb{S}^m} x d\mu_{p,\varepsilon} = (0, \delta - 1)
  \in \R^m\times\R.
\end{equation}
Then $\mu_{p,\varepsilon}(\mathbb{S}^m) \to w$ and up to a subsequence,
$\mu_{p,\varepsilon} \overset{w^*}{\to} \mu_p$.
The measure $\mu_p$ still
satisfies~\eqref{eq:balance}. In particular, it implies that the support of
$\mu_p$ has
at least two points, i.e.
\begin{equation}\label{ineq:blow-up-alternative}
  \text{either } \mu_p^c \neq 0 \text{ or } \abs{A(\mu_p)} \ge 2
  \text { on } \R^m\cup\brc{\infty} \approx \Sph^m
\end{equation}
Let us decompose
\begin{equation}
  \mu_p = \mu_p^c + w_N \delta_{N} + \sum_{p' \in A(\mu_p)\cap \R^m} w_{p'}\delta_{p'}.
\end{equation}
Because of~\eqref{eq:balance}, we see that
$\frac{w_N  - (w - w_N)}{w} \leq \delta - 1$,
so $w_N \leq \frac{1}{2}w\delta$ and therefore
\begin{equation}
  w\geq \mu_p(\R^m) \geq w(1 - \delta).
\end{equation}
\begin{remark}\label{rem:supp-on-bndry}
  Notice that if $\supp \mu_\epsilon \subset \partial M$
  (the case of the Steklov problem), then $p_\epsilon \in
  \mathbb{R}^{m-1}$, $T^p_\epsilon \colon U \to \R^m_+$ and
  $\supp \mu_{p,\epsilon} \subset \mathbb{S}^{m-1} \subset \mathbb{S}_{+}^{m}$.
\end{remark}

\begin{claim}
  The sets $T^p_\epsilon[B(r_\epsilon)]$ exhaust
  $\mathbb{R}^m$.
\end{claim}
\begin{proof}
  It is obvious that $p_\varepsilon \in B(r_\epsilon)$. Then
  it is enough to prove that
  $\frac{r_\epsilon - \abs{p_\varepsilon}}{\alpha_\varepsilon}
  \to \infty$.
  Since the measures $\mu_\epsilon$ concentrate on $B(r^2_\epsilon)$, we see that
  the origin $\brc{0} \in \R^m$ must be an accumulation point of the sets
  \begin{equation}
    B(\alpha^{-1}_\epsilon r^2_\epsilon)-\alpha^{-1}_\epsilon p_\epsilon,
  \end{equation}
  or else we would separate the origin and the sets
  by hyperplanes $H_\epsilon$ such that $\dist(H_\epsilon, 0) \geq c_1$,
  and the center of mass of $\mu_{p,\epsilon}$ would not lie
  on the last coordinate axis. In other words, the equation~\eqref{eq:balance}
  would not be possible for small $\epsilon$.

  In particular, this means
  \begin{equation}
    \alpha^{-1}_\epsilon\abs{ p_\epsilon} \leq \alpha^{-1}_\epsilon r^2_\epsilon
    +o(1).
  \end{equation}
  Therefore
  \begin{equation}
    \alpha^{-1}_\epsilon(r_\epsilon - \abs{p_\varepsilon})
    \geq \alpha^{-1}_\epsilon r_\epsilon(1 - r_\epsilon)
    - o(1),
  \end{equation}
  and obviously $\alpha^{-1}_\epsilon r_\epsilon \to \infty$,
  or else
  $B(\alpha^{-1}_\epsilon r^2_\epsilon)-\alpha^{-1}_\epsilon p_\epsilon$
  would converge to a point contradicting~\eqref{eq:balance}.
\end{proof}
Since $T^p_\epsilon[B(r_\epsilon)]$ exhaust $\mathbb{R}^m$,
it is clear that
\begin{equation}
  (T^p_\epsilon)_*[\mu_\epsilon] \overset{w^*}{\to} \mu_p
  \text{ on } \R^m.
\end{equation}
Now, consider the following
logarithmic cut-off function $\eta_{r,R} \colon M \to [0,1]$,
\begin{equation}\label{eq:log-cut-off}
  \eta_r^R(x) = \begin{dcases}
    0,                                       & \quad x \in M \setminus B(R,p)      \\
    \frac{\ln \frac{\abs{x}}{R}}{\ln \frac{r}{R}}, & \quad x \in B(R,p) \setminus B(r,p) \\
    1,                                       & \quad x \in B(r,p)
  \end{dcases},
\end{equation}
\begin{equation}\label{eq:cut-off-estim}
  \int \abs{d\eta^R_r}^m_g dv_g \sim \brr{\frac{1}{\log \frac{R}{r}}}^{m-1},
\end{equation}
where $ m = \dim M$.

\paragraph{Finiteness of the bubbling tree.}
For every $t > 0$, we are going to construct a sequence of
test functions with the Rayleigh quotient~\eqref{eq:variat-character}
converging to $0$ as $t\to0$.
For every atom $p \in A(\mu) \subset M$ we can pick
$\phi^p_\epsilon := \eta_{r^2_t}^{r_t}$ for some small enough $r_t$
to see that
\begin{equation}\label{eq:vanish-prop}
  \lim_{\epsilon\to 0}\int (\phi^p_t)^2 d\mu_\epsilon \geq \frac{1}{2}w_p
  \text{  and  } \int \abs{d\phi^p_t}^2 dv_g < t.
\end{equation}
At the same time, if $\mu^c \neq 0$, one has an extra
test function
\begin{equation}\label{eq:vanish-prop-2}
  \phi^c_t = 1 -\sum_{p \in A(\mu)} \eta^{\sqrt{r_t}}_{r_t},
  \quad \lim_{\epsilon\to 0}\int(\phi^c_t)^2 d\mu_\epsilon \geq \frac{1}{2}\mu^c(M)
\end{equation}
Then we can decompose every $\phi^p_t$
into
\begin{equation}\label{eq:atom-decomp}
  \phi^p_t =
  \phi^p_t\brr{\phi^{p,\infty}_{t}\circ T^p_\epsilon}
  +\phi^p_t\brr{\phi^{p,c}_{t}\circ T^p_\epsilon}
  +\sum_{p' \in A(\mu_p)\cap\R^m}
  \phi^p_t\brr{\phi^{p'}_t \circ T^p_\epsilon},
\end{equation}
where all the summands have disjoint compact supports,
satisfy~\eqref{eq:vanish-prop} or~\eqref{eq:vanish-prop-2} with $w_{p'}$ and
\begin{itemize}
  \item $\phi^{p,\infty}_{t}  = 0$ if $w_\infty := w_{N} = 0$,
  or else
  we set $\phi^{p,\infty}_{t} := 1 - \eta^{R^3_t}_{R^2_t}$,
  for some large enough $R_t$.
  \item $\phi^{p,c}_{t} = 0$ if
  $\mu_p^c = 0$, or else we set
  $\phi^{p,c}_{t} := \eta^{R^2_t}_{R_t}
  - \sum_{p' \in A(\mu_p)} \eta^{\sqrt{r_t}}_{r_t}$.
\end{itemize}
The estimates~\eqref{eq:vanish-prop} and~\eqref{eq:vanish-prop-2}
are possible because of the equality
\begin{equation}
  I_t := \lim_{\epsilon\to0}\int_{M}
  \abs{d(\phi^{p,\infty}_{t}\circ T^p_\epsilon)}_g^2
  dv_g = \brr{\lim_{\epsilon\to0} \alpha_\epsilon^{m-2}} \int_{\R^m}
  \abs{d\phi^{p,\infty}_{t}}^2 dx,
\end{equation}
i.e. either $m> 2$ and $I$ is trivially zero, or $m=2$ and
one can make $I_t$ arbitrarily small by~\eqref{eq:cut-off-estim}.
Then we can blow up the
atoms of $\mu_p$ on $\R^m$ so that
\begin{equation}
  T^{i}_\epsilon := (T^{p'}_\epsilon\circ T^p_\epsilon)_*[\mu_\epsilon]
  \overset{w^*}{\to} \mu_{p'} \text{ on } \R^m, \text{ where }
  p' \in A(\mu_p).
\end{equation}
Notice that
\begin{equation}\label{eq:composition}
  T^{p'}_\epsilon\circ T^p_\epsilon =
\frac{x - p_\epsilon - \alpha_\epsilon p'_\epsilon}
{\alpha_\epsilon\alpha'_\epsilon}=:
\frac{x - \tilde{p}_\epsilon}
{\tilde{\alpha}_\epsilon}
\end{equation}
for some $\tilde{p}_\epsilon \in M$. Iteratively applying the
procedure described above, we arrive at
the following proposition.

\begin{proposition}\label{prop:bubble-tree-existence}
  Let $\brc{\mu_\varepsilon} \subset \mathcal{M}^+_c(M)$,
  $\mu_\varepsilon \overset{w^*}{\to} \mu$,
  and $\lambda_k(\mu_\epsilon) \geq c > 0$. Then for any $\delta > 0$,
  $\brc{\mu_\varepsilon}$ has the $\delta$-bubbling property
  with at most $k$ maps $\brc{T^i_\epsilon}$. If
  $\supp \mu_\epsilon \subset \partial M$, one has
  $T^i_\epsilon \colon B(1,p_i) \to \R^m_+$.
\end{proposition}
\begin{proof}
  Pick $t < c$. Arguing inductively as in the beginning
  of the paragraph, at each iteration step $n$, we have a collection of
  $k_n$ test functions $\brc{\phi_i} \subset C(M)\cap H^1(M)$ with disjoint supports
  satisfying~\eqref{eq:vanish-prop} or~\eqref{eq:vanish-prop-2}. Because
  of~\eqref{ineq:blow-up-alternative} and~\eqref{eq:atom-decomp},
  one has
  \begin{equation}
    k_{n+1} > k_n.
  \end{equation}
  So, once $k_n > k$, by variational characterization
  we will have $c \leq \lambda_k(\mu_\epsilon) < t$,
  which contradicts
  the assumptions of the proposition.

  Therefore, the iteration process stops at some point and one
  can check that at every step, the
  conditions~\eqref{eq:dif-rate-condition}, \eqref{ineq:alomst-no-loss},
  and~\eqref{eq:composition} are satisfied. Remark~\ref{rem:supp-on-bndry}
  justifies the last statement.
\end{proof}

\subsection{The case of \texorpdfstring{$\dim M = 2$}{Lg}}

Let us use the convention
$\overline{\lambda}_k(\mu) = 0$ if $\mu = 0$. We use isothermal
coordinates since $\dim M = 2$.
\begin{corollary}\label{cor:union-ineq}
  Let $\brr{M,g}$ be a closed Riemannian surface, $\partial M = \varnothing$, and
  $\brc{\mu_\epsilon} \subset \mathcal{M}^+_c(M)$ be a
  sequence of (continuous) Radon measures such that
  $\mu_\epsilon \oset{w^*}{\to} \mu \neq 0$ and
  $\lambda_k(\mu_\epsilon) \geq c > 0$.
  Denote by $\mu^c$ the continuous part of $\mu$. Then for any
  $\delta > 0$,
  there exist at most $k$ measures $\brc{\tilde{\mu}_i}_{1\leq i \leq b}
  \subset \mathcal{M}^+_c(\mathbb{S}^2)$,
  and numbers $k_i \in \mathbb{N}$ such that
  \begin{equation}
    \mu(M)(1 - \delta) \leq \mu^c(M) + \sum_{i= 1}^b \tilde{\mu}_i(\mathbb{S}^2) \leq \mu(M)
  \end{equation}
  and
  \begin{equation}
    \limsup_{\epsilon\to 0}\overline{\lambda}_k(\mu_\epsilon)
    \leq (1 + \delta) \brr{\overline{\lambda}_{k_0}\brr{[g],\mu^c} +
    \sum_{i=1}^b \overline{\lambda}_{k_i}\brr{\mathbb{S}^2,\tilde{\mu}_i}},
  \end{equation}
  where $k_0 + \sum_{i = 1}^b k_i = k$ for some $k_0 \in \mathbb{N}$.
\end{corollary}
\begin{proof}
  Fix $\delta > 0$ such that  $\brr{1-\delta}^{-1} < 1 + \delta$.
  Then by Proposition~\ref{prop:bubble-tree-existence},
  the sequence has the $\delta$-bubbling property.
  Let $\tilde{\mu}_i = \mu^c_i$, be the continuous
  part of the bubble measure $\mu_i$. Set also
  $\tilde{\mu}_0 := \mu^c$ and denote
  \begin{equation}
    \tilde{\mu} = \tilde{\mu}_0 \sqcup \cdots \sqcup \tilde{\mu}_b.
  \end{equation}

  By the variational characterization of $\lambda_k\brr{\tilde{\mu}}$,
  one can find $k+1$ linearly independent functions
  $\tilde{\phi} \in \brc{\tilde{\phi}_0,\cdots, \tilde{\phi}_k} \subset C^\infty(\tilde{M})$ on
  the disjoint union $(\tilde{M}, \tilde{\mu}) := (M,\mu^c)
  \sqcup (\Sph^2,\tilde{\mu}_1) \sqcup \cdots \sqcup (\Sph^2,\tilde{\mu}_b)$
  such that
  \begin{equation}
    \int_{\tilde{M}} \tilde{\phi}^2 d\tilde{\mu} = 1
    \text{  and  }
    \int_{\tilde{M}} \abs{d\tilde{\phi}}^2_{\tilde{g}} dv_{\tilde{g}}
    < \lambda_k\brr{\tilde{\mu}} +  \delta',
  \end{equation}
  where $\delta'$ will be chosen later.
  By using logarithmic cut-offs~\eqref{eq:log-cut-off},
  one can think that $\tilde{\phi}|_M \in
  {C^\infty_0(M\setminus A(\mu))}$ and
  $\tilde{\phi}|_{\Sph_i} \in
  C^\infty_0(\R^2\setminus A(\mu_i))$. Then by using the maps
  $T^i_\epsilon \colon U_i \to \Sph^2$,
  one can descend the functions $\tilde{\phi}$ to
  the $C^\infty_0$-functions $\brc{\phi_0,\cdots, \phi_k}$ on $M$.
  When $\epsilon$ is small enough, they are linearly
  independent as well. Since
  $(T^i_\epsilon )_*[\mu_\epsilon] \overset{w^*}{\to} \mu_i$,
  and the expression $\abs{d\phi}^2_{g} dv_{g}$ is conformally
  invariant, one has
  \begin{equation}
    \int_{M} \phi^2 d\mu \to 1
    \text{  and  }
    \int_{M} \abs{d\phi}^2_{g} dv_{g}
    < \lambda_k\brr{\tilde{\mu}} +  \delta',
  \end{equation}
  which implies
  \begin{equation}
    \limsup_{\epsilon\to 0}\lambda_k(\mu_\epsilon) \le
    \lambda_k\brr{\tilde{\mu}} + \delta'.
  \end{equation}

  Let us recall that the spectrum of a disjoint union of
  manifolds is the union of the spectra of
  the connected components. In particular, if it happens
  that $\lambda_{k_i-1}(\tilde{\mu}_i) < \lambda_k\brr{\tilde{\mu}} < \lambda_{k_i}(\tilde{\mu}_i)$
  for some $k_i \in \mathbb{N}$, one can rescale $\tilde{\mu}_i$
  by $c_i \geq 1$, so that
  $\lambda_{k_i}(c_i\tilde{\mu}_i) = \lambda_k\brr{\tilde{\mu}}$.
  In all other cases, we set $c_i = 1$ and $\lambda$ is
  already an eigenvalue of $\tilde{\mu}_i$, and
  we choose $k_i$ in such a way that equation
  $\sum k_i = k$ is satisfied. It is possible since $\lambda_k\brr{\tilde{\mu}}$
  remains the $k$-th eigenvalue of $c_0\tilde{\mu}_0 \sqcup \cdots \sqcup c_b\tilde{\mu}_b$.
  Thus, we see that
  \begin{equation}
   \begin{gathered}
    \limsup_{\epsilon\to 0}\overline{\lambda}_k(\mu_\epsilon)
    \leq \delta'\mu(M) + \mu(M)\lambda_k\brr{\tilde{\mu}}
    = \delta'\mu(M) + \frac{\mu(M)}{\tilde{\mu}(M)}\tilde{\mu}(M)\lambda_k\brr{\tilde{\mu}}
    \\\leq\delta'\mu(M) + \brr{1-\delta}^{-1}
    \sum_i \lambda_k\brr{\tilde{\mu}} \int d(c_i \tilde{\mu}_i)
    < \delta'\mu(M) + (1+\delta)\sum_i \overline{\lambda}_{k_i}(\tilde{\mu}_i).
   \end{gathered}
  \end{equation}
  since $\brr{1-\delta}^{-1} < 1+\delta$. Now, we choose $\delta'$
  so that $\delta'\mu(M)$ is absorbed by the second term.
\end{proof}

\begin{corollary}\label{cor:st-union-ineq}
  Let $\brr{M,g}$ be a compact Riemannian surface with boundary and
  $\brc{\mu_\epsilon} \subset \mathcal{M}^+_c(\partial M)$ be a
  sequence of (continuous) Radon measures such that
  $\mu_\epsilon \oset[0.32ex]{w^*}{\to} \mu \neq 0$ and
  $\lambda_k(\mu_\epsilon) \geq c > 0$.
  Denote by $\mu^c$ the continuous part of $\mu$. Then
  there exist at most $k$ measures $\brc{\tilde{\mu}_i}_{1\leq i \leq b} \in \mathcal{M}^+_c(\partial\mathbb{D}^2)$,
  and numbers $k_i \in \mathbb{N}$ such that
\begin{equation}
  \mu(\partial M) - \delta \leq\mu^c(\partial M) +
    \sum_{i=1}^b \tilde{\mu}_i(\partial \mathbb{D}^2) \leq \mu(\partial M)
\end{equation}
  and
  \begin{equation}
    \limsup_{\epsilon\to 0}\overline{\lambda}_k(\mu_\epsilon)
    \le (1+\delta)\brr{\overline{\lambda}_{k_0}\brr{[g],\mu^c} +
    \sum_{i=1}^b \overline{\lambda}_{k_i}\brr{\mathbb{D}^2,\tilde{\mu}_i}},
  \end{equation}
  where $k_0 + \sum_{i=1}^b k_i = k$ for some $k_0 \in \mathbb{N}$.
\end{corollary}
\begin{proof}
  The proof is the same as in the previous corollary. By
  Remark~\ref{rem:supp-on-bndry}, the bubbling construction
  produces measures $\mu_{i}$ with supports
  on the geodesic circle
  $\overline{\pi(\mathbb{R}\times\brc{0})}$ if we initially have
  $\supp \mu_\epsilon \subset \partial M$, and
  the map $T^i_\epsilon \colon B(1, p^i) =
  B_+(r_\epsilon, p) \to \mathbb{S}^2_{+} \approx
  \mathbb{D}^2$ maps half balls to half spheres, which
  are conformally equivalent to disks.
\end{proof}

\medskip
\printbibliography
\end{document}